\documentclass{tran-l}
\usepackage{amsmath}
\usepackage{amssymb}
\usepackage{graphicx}
\usepackage[all,cmtip]{xy}
\RequirePackage{lineno}

 \newtheorem{theorem}{Theorem}[section]
 \renewcommand\thetheorem{\Alph{theorem}}
 \newtheorem{corollary}[theorem]{Corollary}
 \newtheorem{lemma}[theorem]{Lemma}
 \newtheorem{proposition}[theorem]{Proposition}
 \newtheorem{definition}[theorem]{Definition}
 \newtheorem{remark}[theorem]{Remark}
 \newtheorem{example}[theorem]{Example}
 \newtheorem{observation}[theorem]{Observation}
 
 \numberwithin{equation}{section}

\begin{document}

\title[Removal of Singularities]
{A Geometric Proof of Removal of Boundary Singularities of Pseudo-Holomorphic Curves}

\author{URS FUCHS}
\address{Urs Fuchs, Mathematisches Institut, Westf\"{a}lische
Wilhelms-Universit\"{a}t M\"{u}nster, 48149 M\"{u}nster, Germany
\newline
\indent Department of Mathematics, Purdue University, West Lafayette, IN
47907, USA}
\email{ufuchs@math.purdue.edu}

\author{LIZHEN QIN}
\address{Lizhen Qin, Department of Mathematics, Purdue University, West Lafayette, IN 47907, USA}
\email{qinl@math.pudue.edu}








\begin{abstract}
We prove two theorems on the removal of singularities
on the boundary of a pseudo-holomorphic curve. In one theorem, we need \textit{no} apriori
assumption on the area of the curve. The proof uses a doubling argument with the goal of converting curves
with boundary to curves without boundary. Our method is new and geometric and it does \textit{not}
need Sobolev spaces and PDEs.
\end{abstract}

\maketitle

\section{Introduction}
We provide in this paper a new and detailed geometric proof of the removal of boundary singularities
of pseudo-holomorphic curves. More precisely,
we shall prove the following Theorems \ref{theorem_A} and \ref{theorem_B} due to Gromov \cite{gromov}
by using a doubling argument.
The idea of using a doubling argument is also due to him.
Precise statements of Theorems \ref{theorem_A} and \ref{theorem_B} are Theorems
\ref{theorem_finite_area} and \ref{theorem_tame} respectively.

Let $D^{+}$ be the upper half disk on the complex plane and $\check{D}^{+}$ be
the punctured upper half disk $D^{+} - \{0\}$ (see (\ref{half_disk}) and (\ref{pucture_half_disk})).
Suppose $M$ is a manifold with almost complex structure $J$ and $W$ is an embedded totally
real submanifold of $M$. Assume $f: \check{D}^{+} \rightarrow M$ is a smooth and pseudo-holomorphic map such that
$f(\partial \check{D}^{+}) \subseteq W$.

\begin{theorem}\label{theorem_A}
Suppose the image of $f$ is relatively compact
and the area of the image is finite (see Definition \ref{definition_image_volume})
for some Riemannian metric.
Then $f$ has a smooth extension over $0 \in D^{+}$.
\end{theorem}

\begin{theorem}\label{theorem_B}
Suppose the image of $f$ is relatively compact.
Suppose $\alpha$ is a $1$-from on $M$ such that $d \alpha$ tames $J$.
Furthermore, assume $\alpha$ is exact on $W$. Then $f$ has a smooth extension over $0 \in D^{+}$.
\end{theorem}

In Theorem \ref{theorem_B}, the assumption of the exactness of $\alpha$ on $W$
was not stated by Gromov in \cite[1.3.C']{gromov}. However,
Theorem \ref{theorem_B} will \textit{no} longer be true if we drop this assumption.
We construct the following simple counterexample to show that, without the assumption,
there could be \textit{no} continuous extension. (The proof of this example is given
in Section \ref{section_main_result}.)

\renewcommand\thetheorem{\arabic{theorem}}
\numberwithin{theorem}{section}

\begin{example}\label{example_no_extension}
Let $M = \mathbb{C}$ be the complex plane. Choose $J$ to be the standard complex structure.
Let $W= \{ z=x+iy \in \mathbb{C} \mid |z|=1 \}$ be the unit circle. Define $\alpha = xdy$.
Define a holomorphic function $f: \check{D}^{+} \rightarrow \mathbb{C}$ as
\begin{equation}
f(z) = \exp \left( -\frac{i}{z} \right).
\end{equation}

Then $f: \check{D}^{+} \rightarrow (M, W, J, \alpha)$ satisfies every assumption
in Theorem \ref{theorem_B} with one exception: $\alpha$ is not exact on $W$.

The conclusion of Theorem \ref{theorem_B} does not hold in this example.
\end{example}

To the best of our knowledge, up to now, there has been no work in the literature to
give a correct statement of Theorem \ref{theorem_B}, let alone to prove it.
Theorem \ref{theorem_B} does have certain advantages over the Theorem \ref{theorem_A}.
This will be illustrated by a simple Example \ref{example_advantage}.

We shall call the above two theorems the \textit{boundary case} of the removal of singularities.
In these theorems, if $D^{+}$ (resp. $\check{D}^{+}$) is replaced by the disk
$D =  \{ z \in \mathbb{C} \mid |z|<1 \}$ (resp. the punctured
disk $\check{D} = D - \{0\}$) and all assumptions related to $W$ are dropped, we will get theorems
on the removal of interior singularities (see \cite[1.3 \& 1.4]{gromov} and \cite[p.41, Theorem 2.1]{hummel}).
We shall call them the \textit{interior case}.

The removal of singularities is important to symplectic geometry.
It is a key ingredient in Gromov compactness for pseudo-holomorphic curves.
Namely it enables us to identify pseudo-holomorphic
spheres and disks as the obstructions to
compactness for pseudo-holomorphic curves. (See \cite[1.5]{gromov} and \cite[p.71]{mcduff_salamon}.)

There has been much work in the literature to present detailed proofs of the removal of singularities
in both interior case and boundary case. They can be roughly divided into two types.

The first type is analytic. It involves nontrivial tools of analysis such as Sobolev spaces
and PDEs. The strategy is as follows. One first proves the pseudo-holomorphic map $f$ on $D$ or $D^{+}$
belongs to the Sobolev space $W^{1,p}$ for some $p>2$, or is H\"{o}lder continuous. Here one
improves the integrability of the derivative of $f$ by analytic methods. Then the elliptic
regularity of PDEs implies that $f$ is smooth on $D$ or $D^{+}$. This method has been used in, e.g.
\cite{oh}, \cite{parker_wolfson}, \cite{ye}, \cite{mcduff_salamon} and \cite{ivashkovich_shevchishin}.

The second one is geometric and stems from the original argument by Gromov \cite{gromov}. Rather than
using the above analytic machinery, it is based on the geometric insight into the problem. For example,
different geometric aspects yield different types of isoperimetric inequalities. The proof relies on
the combination and the refinement of these isoperimetric inequalities. The process of this proof is different
from that of the analytic one. First, one proves that $f$ has a continuous extension over $0$.
Second, one proves that $f$ is Lipschitz continuous at $0$. Then, by a geometric construction,
one reduces the study of the derivatives of $f$ to the case of $f$ itself. Thus the $C^{k}$
regularity implies the $C^{k+1}$ regularity, which completes the proof. This method
has been used in, e.g. \cite{muller} and \cite{hummel} in the interior case.

Up to now, there has been no geometric proof of the removal
of boundary singularities. This paper provides such a geometric proof.
Therefore, our proof is new for both Theorems \ref{theorem_A} and \ref{theorem_B}.
The readers of this paper do not need any knowledge of Sobolev spaces and PDEs.

Here is a bibliography of the previous work in the boundary case. The paper \cite{oh} gives a proof
of Theorem \ref{theorem_A} under the additional assumption that $J$ is compatible with a symplectic
form and $W$ is Lagrangian. A second proof
of Theorem \ref{theorem_A} is given in \cite{ye}.
The book \cite[Section 4.5]{mcduff_salamon} presents a proof of Theorem \ref{theorem_A} under the additional assumption
that $J$ is tamed by a symplectic form and $W$ is Lagrangian.
The paper \cite[Corollary 3.2]{ivashkovich_shevchishin}
proves a more general theorem which implies Theorem \ref{theorem_A}, where $W$
is not necessarily embedded and the assumption on the area of the image is weakened.

We describe now the main idea of our proof. In \cite[1.3.C]{gromov}, Gromov suggests that a doubling
argument would reduce the boundary case to the interior case. Our method follows this idea.

We try to find a good doubling map $F: \check{D} \rightarrow (M,J)$ of $f: \check{D}^{+} \rightarrow (M,J)$.
The good map $F$ is an extension of $f$
and satisfies two conditions: (1) it is pseudo-holomorphic; (2) it has sufficient symmetry.
Because of the symmetry, one expects that $F$ has good properties if $f$ does. If one can find such $F$,
then the boundary case can be easily reduced to the interior case. It turns out that such a good map
is difficult to be obtained. However, fortunately, we can construct in this paper
a doubling map close to such a map.

A natural way to define a doubling map is as follows: If the image of $f$ is contained in a tubular neighborhood of $W$,
one can reflect the image of $f$ with respect to $W$. More generally, if $f$ maps a smaller punctured half disk into this tubular neighborhood,
one can apply the reflection to the restriction of $f$ to this smaller punctured half disk.
Since the removal of singularities is a local argument, this is sufficient.

There are two difficulties when we do a doubling argument.
First, under the assumption of Theorem \ref{theorem_B}, it's not easy to show
that $f$ maps a smaller punctured half disk into a tubular neighborhood of $W$. This makes the construction
of a doubling map difficult. Actually, Lemma \ref{lemma_converge_neighborhood} shows
that such a situation becomes better in the case of Theorem \ref{theorem_A}.
Second, the triple $(M,W,J)$ lacks sufficient symmetry.
In fact, the symmetry of $(M,W,J)$ helps a doubling argument.
Example \ref{example_pansu} tells us, if one makes the assumption that $J$ is integrable
in a neighborhood of $W$ and $W$ is real analytic, then a holomorphic doubling map
is easily obtained. This strong assumption actually gives $(M,W,J)$ more symmetry.
(This assumption appears in \cite[2.1.D]{gromov} and \cite[p.244-245]{pansu}.)

In order to overcome the first difficulty, we construct the following intrinsic doubling.
We pull the geometric data on $M$, such as metrics and forms, back to $\check{D}^{+}$
by using $f$, and then we extend these data symmetrically over $\check{D}$. Instead of constructing
a doubling map from $\check{D}$ to $M$, we take $\check{D}$ as our ambient manifold. On $\check{D}$,
we can adapt certain arguments which were used by Gromov on $M$ in the interior case.
This reduces Theorem \ref{theorem_B} to Theorem \ref{theorem_A}.
To make these arguments possible, the extended geometric data on $\check{D}$
need to have good properties. We achieve this by two ingredients: the first one is the fact that
the $1$-form $\alpha$ is exact on $W$; the second one is some constructions suggested by Gromov
\cite[1.3.C]{gromov} such as finding a Hermitian metric on $M$ which makes $W$ totally geodesic.

The above intrinsic doubling is not sufficient because we have to use an extrinsic property
of a pseudo-holomorphic curve: it is ``almost minimal" in the ambient manifold (see
comment before Lemma \ref{lemma_isoperimetric_disk}).
Thus we also establish an extrinsic doubling.
Here we do define a doubling map $F: \check{D} \rightarrow M$ by a reflection as mentioned above.
However, this map $F$ is not necessarily pseudo-holomorphic because of the second difficulty
mentioned above. This difficulty is overcome by the following symmetric construction.
We introduce a new almost complex structure $\widetilde{J}$ arising naturally from the
reflection. Then $F$ is pseudo-holomorphic with respect to $J$ on $\check{D}^{+}$ and
with respect to $\widetilde{J}$ on $\check{D}^{-}$ (the lower half disk). Furthermore, $J$ and $\widetilde{J}$
coincide on $W$. Therefore, our map $F$ is sufficiently close to a pseudo-holomorphic map so that
an adaptation of Gromov's approach in the interior case finishes the proof.

The outline of this paper is as follows. Section \ref{section_main_result} precisely formulates the main results
of this paper. Section \ref{section_set_up} lists some technical results
frequently used throughout this paper. In Section \ref{section_reduction_lemma}, we construct
our intrinsic doubling which reduces Theorem \ref{theorem_B} to Theorem
\ref{theorem_A}. The subsequent sections constitute a progress of improving the
regularity of $f$ at $0 \in D^{+}$ under the assumption of Theorem \ref{theorem_A}.
Section \ref{section_pre-continuity} shows that $f$ maps a smaller punctured half disk into a tubular neighborhood
of $W$. In this part, we follow \cite[p.135-136]{oh}. Such a result paves the way for the extrinsic doubling
in the next section. In Section \ref{section_continuity}, we establish the continuous
extension of $f$ over $0$ by using the extrinsic doubling. The key argument in this step
is to establish certain isoperimetric inequalities. This follows from the ``almost minimal"
property of pseudo-holomorphic curves. In Section \ref{section_lipschitz_continuity},
the Lipschitz continuity of $f$ is proved. The proof is a refinement of that in Section
\ref{section_continuity}. It relies on an improvement of the previous isoperimetric inequalities.
Section \ref{section_almost_structure} recalls the fact that the tangent bundle of $M$ is
naturally an almost complex manifold. This is needed for a geometric construction in the
next section. Finally, in Section \ref{section_higher_order_derivatives}, we study the higher
derivatives of $f$ by a boot-strapping argument, which finishes the proof of Theorems
\ref{theorem_A} and \ref{theorem_B}.


\renewcommand\thetheorem{\arabic{theorem}}
\numberwithin{theorem}{section}

\section{Main Result}\label{section_main_result}
In this paper, all manifolds are without boundary if we don't say this
explicitly. All manifolds,
maps, functions, metrics, almost complex structures, forms and so on are smooth
if we don't state this
explicitly. Similarly, all submanifolds are assumed to be smoothly embedded submanifolds unless otherwise mentioned. We say
a submanifold is a closed submanifold if it is a \textit{closed subset} of the ambient manifold.

\begin{definition}\label{definition_image_volume}
Suppose $N$ is an $n$-dimensional $C^{1}$ Riemannian manifold, $S$ is a $k$-dimensional
$C^{1}$ manifold, and $h: S \rightarrow N$ is a $C^{1}$ map.
Pulling back the Riemannian metric from $N$ to $S$ by
$h$, we get a possibly singular $C^{0}$ metric on $S$. The volume of the image of $h$ is
defined as the volume of $S$ with respect to this pull back metric. Denote this volume
by $|h(S)|$. When $k=1$ or $2$, we also call it length or area of the image respectively.
\end{definition}

As a subset of $N$, $h(S)$ has its $k$-dimensional Hausdorff measure with respect to
the metric of $N$. By Federer's area formula, we know that $|h(S)|$ is no less than the
Hausdorff measure of $h(S)$. It also could happen that $|h(S)|$ is strictly greater than
the Hausdorff measure, for example, when $h$ is a covering map.

\begin{definition}
We say a subset $Y$ is relatively compact in a topological space $X$ if the closure of
$Y$ in $X$ is compact.
\end{definition}

Let's recall some basic definitions related to almost complex manifolds. Suppose
$M$ is a manifold with an almost complex structure $J$, that is a field
of endomorphisms on $T_{p} M$ for all $p \in M$ such that $J^{2} = - \text{Id}$.
We call $M$ an almost complex manifold and also denote it by $(M,J)$. The dimension
of $M$ has to be even. We say a $2$-form $\omega$ on $M$ is a symplectic form if $\omega$
is closed and nondegenerate.

\begin{definition}
We say a $2$-form $\omega$ tames $J$ if $\omega(v,Jv) > 0$ for any nonzero
tangent vector $v$ on $M$. We say $\omega$ is compatible with $J$ if $\omega(\cdot, J \cdot)$
is a Riemannian metric on $M$.
\end{definition}

Clearly, the fact that $\omega$ is compatible with $J$ implies that $\omega$ tames $J$.
If $\omega$ tames $J$, then $\omega$ is obviously nondegenerate.
In the compatible case, $J$ preserves the Riemannian metric
$\omega(\cdot, J \cdot)$, and $H(\cdot, \cdot) = \omega(\cdot, J \cdot) - i \omega(\cdot, \cdot)$
defines a Hermitian metric on $M$ with respect to $J$,
where $i$ is the imaginary unit, i.e.\ $i^{2}=-1$.

\begin{definition}\label{definition_totally_real}
A submanifold $W$ of $(M,J)$ is said to be totally real if
$\dim (W) = \frac{1}{2} \dim (M)$ and $J(T_{p}W) \cap T_{p}W
= \{0\}$ for all $p \in W$.
\end{definition}

\begin{definition}
Suppose $(S, J_{1})$ and $(N, J_{2})$ are manifolds with almost complex structures $J_{1}$
and $J_{2}$. A $C^{1}$ map $h: S \rightarrow N$ is $J$-holomorphic or pseudo-holomorphic
if the derivative of $h$ is complex linear with respect to $J_{1}$ and $J_{2}$, i.e.
\[
J_{2} \cdot dh = dh \cdot J_{1}.
\]
If $S$ is a Riemann surface, we call such a map a $J$-holomorphic curve or a
pseudo-holomorphic curve in $N$.
\end{definition}

We fix now some notations for some subsets of the complex plane $\mathbb{C}$ which are frequently
used throughout this paper. Denote by
\begin{equation}\label{half_disk}
D^{+} = \{ z \in \mathbb{C} \mid |z|<1, \text{Im}z \geq 0 \},
\end{equation}
the half disk on the complex plane. Define
\begin{equation}\label{pucture_half_disk}
\check{D}^{+} = D^{+} - \{0\}.
\end{equation}
as the punctured half disk.
Clearly, $\check{D}^{+}$ is a Riemann surface with boundary. Its boundary is
\begin{equation}
\partial \check{D}^{+} = \{ z \in \mathbb{R} \mid 0 < |z| < 1 \}.
\end{equation}

The main goal of this paper is to present a new and geometric proof of the following
two theorems due to Gromov \cite[1.3.C]{gromov}. They are the precise versions of
Theorems \ref{theorem_A} and \ref{theorem_B} in the Introduction.

\begin{theorem}\label{theorem_finite_area}
Suppose $(M,J)$ is a smooth almost complex manifold.
Suppose $f: \check{D}^{+} \rightarrow (M,J)$ is a smooth $J$-holomorphic map
such that $f(\partial \check{D}^{+}) \subseteq W$, where $W$ is a smoothly embedded
totally real submanifold and a closed subset of $M$. Suppose the image of $f$ is relatively compact
and the area of the image is finite for some Riemannian metric on $M$.

Then $f$ has a smooth extension over $0 \in D^{+}$.
\end{theorem}

Actually, the finiteness of the area of the image in Theorem \ref{theorem_finite_area} does not depend on the
Riemannian metric (see comment before Claim \ref{claim_metric}).

\begin{theorem}\label{theorem_tame}
Suppose $(M,J)$ is a smooth almost complex manifold.
Suppose $f: \check{D}^{+} \rightarrow (M,J)$ is a smooth $J$-holomorphic map
such that $f(\partial \check{D}^{+}) \subseteq W$, where $W$ is a smoothly embedded
totally real submanifold and a closed subset of $M$. Denote by $\iota: W \hookrightarrow M$
the inclusion of $W$. Suppose $J$ is tamed by $d \alpha$,
where $\alpha$ is a smooth $1$-from on $M$
such that $\iota^{*} \alpha$ is exact on $W$. Suppose the image of
$f$ is relatively compact.

Then $f$ has a smooth extension over $0 \in D^{+}$.
\end{theorem}

Now we give a proof of our counterexample in the Introduction.

\begin{proof}[Proof of Example \ref{example_no_extension}]
We know that $W$ is compact and $d \alpha = dx \wedge dy$ is even compatible with $J$.
However, $\iota^{*} \alpha$
is not exact on the unit circle $W$. Actually, there exists no $1$-form $\beta$ such that
$d \beta$ tames $J$ and $\iota^{*} \beta$ is exact. Otherwise, by Stokes' formula,
there would be no nonconstant compact $J$-holomorphic curve inside $M$ whose boundary lies
on $W$. However, the inclusion of the closed unit disk is such a curve,
which results in a contradiction.

Clearly, $f$ is $J$-holomorphic. By direct computation, we have $f(\partial \check{D}^{+})
\subseteq W$. Furthermore, $|f(z)| \leq 1$ for all $z \in \check{D}^{+}$, which implies that
the image of $f$ is relatively compact.

Nevertheless, the limit $\displaystyle \lim_{z \rightarrow 0} f(z)$ does not exist.
Actually, $0$ is an essential singularity of $\exp (\frac{-i}{z})$.
Therefore, $f$ has no continuous extension over $0$.
\end{proof}

\begin{remark}
Readers are suggested to understand Example \ref{example_no_extension} together with Lemma
\ref{lemma_form_zero}, Example \ref{example_stokes_disk} and Lemma \ref{lemma_isoperimetric_form}.
\end{remark}

As mentioned in the Introduction, the following simple example shows that Theorem \ref{theorem_tame}
has its advantage over Theorem \ref{theorem_finite_area} in certain cases.

\begin{example}\label{example_advantage}
Let $f$ be a complex valued function on $\check{D}^{+}$ such that $f$ is smooth and bounded.
Suppose $f$ is holomorphic in the interior of $\check{D}^{+}$. Suppose $f$ takes real values
on $\partial \check{D}^{+}$. By the Schwarz reflection principle, $f$ extends to a bounded
holomorphic function on $\check{D}$, which implies that $0$ is a removable singularity.

This removal of singularities trivially follows from Theorem \ref{theorem_tame}
by taking $M = \mathbb{C}$, $W = \mathbb{R}$ and $\alpha = x dy$.

However, it's not easy to apply Theorem \ref{theorem_finite_area} because it's
nontrivial to show that the area of the image of $f$ is finite.
\end{example}

\section{Set Up}\label{section_set_up}
In this section, we shall give some notation, definitions and results which are frequently used
throughout this paper.

First, together with (\ref{half_disk}) and (\ref{pucture_half_disk}), we
define some subsets of the complex plane $\mathbb{C}$ as
\begin{equation}
D = \{ z \in \mathbb{C} \mid |z|<1 \}, \qquad
\overline{D} = \{ z \in \mathbb{C} \mid |z| \leq 1 \},
\end{equation}
\begin{equation}\label{lower_half_disk}
D^{-} = \{ z \in \mathbb{C} \mid |z|<1, \text{Im}z \leq 0 \}, \qquad
\check{D}^{-} = D^{-} - \{0\},
\end{equation}
\begin{equation}
D(r) = \{ z \in \mathbb{C} \mid |z|<r \}, \qquad
\check{D}^{+}(r) = \{ z \in D(r) \mid z \neq 0, \text{Im}z \geq 0 \},
\end{equation}
and
\begin{equation}
\partial D(r)
= \{ z \in \mathbb{C} \mid |z|=r \}.
\end{equation}

Let's go back to Definition \ref{definition_image_volume}.
If $S$ is also a $C^{1}$ Riemannian manifold, then the volume of the image,
$|h(S)|$, also equals the Riemannian integral of the (absolute value of the)
Jacobian of $h$ on $S$. In particular, suppose $S$ is a Riemann surface and $N$ is an
almost complex manifold. Suppose both $S$ and $N$ are equipped with Hermitian metrics.
Suppose $h$ is $J$-holomorphic. Then $h$ is conformal and therefore the
Jacobian of $h$ equals $\|dh\|^{2}$.

Suppose $\varphi: D \rightarrow N$ is a $C^{1}$ $J$-holomorphic map from the unit disk
to a $C^{1}$ almost complex manifold $N$. Equip $D$ with the standard Euclidean metric,
and equip $N$ with a $C^{1}$ Hermitian metric. For $0< r \leq 1$, define
\[
A(r) = |\varphi(D(r))|.
\]
For $0< r <1$, define
\[
L(r) = |\varphi(\partial D(r))|.
\]
By the conformality of $\varphi$, we infer
\begin{equation}\label{holomorphic_circle_length}
A(r) = \int_{D(r)} \| d \varphi \|^{2}, \qquad \text{and} \qquad
L(r) = \int_{\partial D(r)} \| d \varphi \|.
\end{equation}

We see that $A(r)$ is $C^{1}$ on $[0,1)$ and $L(r)$
is $C^{0}$ in $[0,1)$. Furthermore, we can easily prove the following formula,
for $r \in (0,1)$, (see the statement and the proof in the last line of \cite[p.\ 315]{gromov})
\begin{equation}\label{gromov_formula}
\frac{d}{dr} A(r) \geq (2 \pi r)^{-1} L(r)^{2}.
\end{equation}
This simple and powerful formula will play a key role in our argument.

We can obviously
make a slight generalization of (\ref{gromov_formula}). Actually, this has been done in the
proof of \cite[1.3.B']{gromov}. Suppose the above $\varphi$ is only defined on $D-\{z_{0}\}$
and $A(1)$ is finite. Then $A$ is continuous in $[0,1]$, which follows from the absolute
continuity of integral. Furthermore, $A$ is $C^{1}$ in $[0, |z_{0}|) \cup (|z_{0}|,1)$,
(here $z_{0}$ could be $0$), and (\ref{gromov_formula}) still holds in the interior of these intervals. This
generalization will be frequently used in this paper.

Suppose $(M,J)$ is an almost complex manifold. Let $H$
be a Hermitian metric on $M$. Then, $\text{Re}H$, the real part of $H$ is a Riemannian metric
on $M$; $-\text{Im}H$, the negative of the imaginary part of $H$ is a nondegenerate $2$-form on $M$.
In order to prove Theorems \ref{theorem_finite_area} and \ref{theorem_tame}, Gromov suggests
the following lemma (\cite[1.3.C]{gromov}). This construction has been used in some
previous proofs (e.g. \cite[Section 4.3]{mcduff_salamon}). It is important for us as well.

\begin{lemma}\label{lemma_hermitian}
Suppose $W$ is a closed totally real submanifold of $(M, J)$.
Then there exists a Hermitian metric $H$ on $M$ which satisfies the following properties.

(1). $TW \perp J(TW)$ with respect to $\text{Re}H$.

(2). $W$ is totally geodesic with respect to $\text{Re}H$.

(3). There exists a $1$-form $\alpha_{0}$ in a neighborhood of $W$ such that
$\iota^{*} \alpha_{0} = 0$ on $W$,
$d \alpha_{0}$ tames $J$ and $d \alpha_{0} = - \text{Im}H$ on $TM|_{W}$,
where $\iota: W \hookrightarrow M$ is the inclusion.
\end{lemma}

\begin{proof}
A proof of (1) and (2) is presented in \cite[Lemma 4.3.3]{mcduff_salamon}.
 One can construct such a $H$ first in a neighborhood $U$ of $W$ and then extend it
(using that $W$ is closed) to a Hermitian metric $H$ defined on $M$ satisfying (1) and (2).

Now we prove (3). By (1), we know $J(TW)$ is the normal bundle of
$W$ inside $M$. By using the exponential map,
a neighborhood of $W$ in $M$ is identified with a neighborhood
of $W$ (i.e.\ the zero section) in $J(TW)$. By using $\text{Re}H$,
we can further identify $J(TW)$ with $T^{*}W$ as follows. For each $v \in J(T_{x}W)$, the map
\[
w \rightarrow \text{Re}H (v, Jw)
\]
is a linear form on $T_{x}W$, where $w \in T_{x}W$.

Therefore, a neighborhood $U$ of $W$ in $M$ is identified with a neighborhood $\mathcal{N}$ of $W$
(i.e.\ the zero section) in $T^{*}W$.

On the manifold $T^{*}W$, there exists a natural $1$-form $\alpha_{0}$:
it is locally written as $-\sum_{j=1}^{n} p_{j} d q_{j}$ in terms of the standard (Position, Momentum)
coordinates $((q_{1}, \cdots, q_{n}),(p_{1}, \cdots, p_{n}))$.
By the identification between $U$ and $\mathcal{N}$,
we have that $\alpha_{0}$ is defined on $U$ and $\iota^{*} \alpha_{0} = 0$. We can also check that,
for $w_{1}, w_{2} \in T_{x}W \subseteq T_{x}U$ and $v_{1}, v_{2} \in J(T_{x}W) \subseteq T_{x}U$,
\begin{eqnarray*}
& & d \alpha_{0} (w_{1} + v_{1}, w_{2} + v_{2}) \\
& = & - \text{Re}H (v_{1}, Jw_{2}) + \text{Re}H (v_{2}, Jw_{1}) \\
& = & - \text{Im}H (v_{1}, w_{2}) - \text{Im}H (w_{1}, v_{2}) \\
& = & - \text{Im}H (w_{1} + v_{1}, w_{2} + v_{2}).
\end{eqnarray*}
We infer, $d \alpha_{0} = - \text{Im}H$, i.e.\ $d \alpha_{0}$ is compatible with
$J$, on $TM|_{W}$. Therefore, $d \alpha_{0}$ tames $J$ in a (possibly smaller) neighborhood of $W$.
\end{proof}

In Theorem \ref{theorem_finite_area}, we assume that the image of $f$ has finite
area with respect to a specific Riemannian metric. Actually, the finiteness of the area
does not depend on the metric: Any two metrics on $f(\check{D}^{+})$ are equivalent since $f(\check{D}^{+})$
is relatively compact. Therefore, we get the following.

\begin{observation}\label{claim_metric}
Without loss of generality, when we prove Theorem \ref{theorem_finite_area},
we may assume that the metric in the assumption of
this theorem is the $\text{Re}H$ in Lemma \ref{lemma_hermitian}.
\end{observation}

Now we point out that we may assume $f$ is an embedding in the assumption of
Theorems \ref{theorem_finite_area} and \ref{theorem_tame} when we prove them.
The idea is to replace $f$ by its graph. More precisely, we do the following
graph construction.

Define
\[
\hat{M} = \mathbb{C} \times M,
\]
which is also an almost complex manifold. Define
$\hat{f}: \check{D}^{+} \rightarrow \hat{M} = \mathbb{C} \times M$
as
\[
\hat{f} (z) = (z, f(z)).
\]
Then $\hat{f}$ is certainly a $J$-holomorphic embedding.
Let
\[
\hat{W} = \mathbb{R} \times W.
\]
Then $\hat{W}$ is a closed totally real submanifold of $\hat{M}$
and $\hat{f}(\partial D^{+}) \subseteq \hat{W}$.

It's easy to see that $\hat{f}: \check{D}^{+} \rightarrow \hat{M}$ satisfies the assumption
of Theorems \ref{theorem_finite_area} and \ref{theorem_tame} as long as $f$ does.
If $\hat{f}$ has a
$C^{k}$ extension over $0$, then certainly so does $f$. Therefore, we get the following.

\begin{observation}\label{claim_embedding}
Without loss of generality,
we may assume $f$ is an embedding in the assumption of Theorems \ref{theorem_finite_area}
and \ref{theorem_tame} when we prove them.
\end{observation}

Now we describe a cone construction which is useful to obtain certain isoperimetric
inequalities for $J$-holomorphic curves.

Let $\Gamma = \gamma(S^{1})$ be a closed curve in a finite dimensional inner product vector space
$V$, where $S^{1}$ is the unit circle and $\gamma: S^{1} \rightarrow V$ is a $C^{1}$ immersion.
Define the center of mass of $\Gamma$ as
\begin{equation}\label{mass_center}
c = \frac{1}{|\Gamma|} \int_{S^{1}} \gamma \|d \gamma\|.
\end{equation}
The center of mass $c$ does not depend on the parametrization of $\Gamma$.
Joining each point in $\Gamma$ to $c$ by a line, we construct a cone with vertex $c$
and boundary $\Gamma$. Denote this cone by $K(\Gamma)$. The following classical
isoperimetric inequality shows a relation between the length of $\Gamma$ and the
area of $K(\Gamma)$.

\begin{lemma}\label{lemma_cone_isoperimetric}
\[
|\Gamma|^{2} \geq 4 \pi |K(\Gamma)|.
\]
\end{lemma}

The book \cite[Appendix A]{hummel} presents a quick proof of a result more general
than Lemma \ref{lemma_cone_isoperimetric}. (Actually, the computation at the top on
\cite[p.\ 116]{hummel} is sufficient for proving this lemma.) So it's safe to omit a proof
here.

\begin{remark}
In the above cone construction, it's important to choose the vertex to be the center of mass $c$
in (\ref{mass_center}). Otherwise, Lemma \ref{lemma_cone_isoperimetric}
will no longer be true.
\end{remark}

\section{A Reduction Lemma}\label{section_reduction_lemma}
The main goal of this section is to prove the following apriori estimate which reduces Theorem \ref{theorem_tame}
to Theorem \ref{theorem_finite_area}.

\begin{lemma}\label{lemma_tame_to_area}
Under the assumption of Theorem \ref{theorem_tame}, equip $M$ with an arbitrary
Riemannian metric and equip $\check{D}^{+}$ with the standard Euclidean metric.
Then there exists a constant $C>0$ such that
\begin{equation}\label{lemma_tame_to_area_1}
\|df(z)\| \leq \frac{C}{|z| \log\frac{1}{|z|}}.
\end{equation}
In particular $|f(\check{D}^{+}(r))| < + \infty$ for all $r\in(0,1)$.
\end{lemma}

Lemma \ref{lemma_tame_to_area} shows that $|f(\check{D}^{+}(r))| < + \infty$ for
any $r \in (0,1)$. Therefore, up to a holomorphic reparametrization of $\check{D}^{+}(r)$, we infer that
$f: \check{D}^{+}(r) \rightarrow M$ in Theorem \ref{theorem_tame} satisfies
the assumption of Theorem \ref{theorem_finite_area}.
(More precisely, the map $f_{r}: \check{D}^{+} \rightarrow M$ satisfies the assumption
of Theorem \ref{theorem_finite_area}, where $f_{r}(z) = f(rz)$. The map
$z \rightarrow rz$ is a holomorphic reparametrization of $\check{D}^{+}(r)$.)
Since the removal of singularities at $0 \in \check{D}^{+}$ is a local argument,
Theorem \ref{theorem_tame} is reduced to Theorem \ref{theorem_finite_area}.

Our strategy to prove (\ref{lemma_tame_to_area_1}) starts from the following observation.
An estimate on $\|df\|$ is equivalent to an estimate on the metric on $\check{D}^{+}$ pulled
back by $f$. To estimate the pullback metric, we need the intrinsic doubling mentioned in the Introduction.

The pullback metric is extended to be a metric $G$ on $\check{D}$. Following Gromov's approach in the
interior case (see \cite[1.3.A']{gromov}), we bound $G$ in terms of the hyperbolic metric on $\check{D}$,
which immediately implies (\ref{lemma_tame_to_area_1}). To bound $G$,
we need to estimate the derivatives of the universal covering maps from the Euclidean disk
$D$ to $(\check{D},G)$.

In this process, we use certain isoperimetric inequalities resulting from the Gaussian curvature and forms on
$\check{D}$ (see Lemmas \ref{lemma_isoperimetric_gauss} and \ref{lemma_isoperimetric_form}).
We shall study the metrics, Gaussian curvature and forms on $\check{D}$ at first.

Similar as Observation \ref{claim_metric}, we may assume the metric in Lemma
\ref{lemma_tame_to_area} is the $\text{Re}H$ in Lemma \ref{lemma_hermitian}. By Observation
\ref{claim_embedding}, we may also assume that $f$ is an embedding. Then the pull back
metric $f^{*}\text{Re}H$ makes sense and is conformal on $\check{D}^{+}$. Therefore, we have
the following two lemmas given by \cite{gromov}.

\begin{lemma}\label{lemma_gauss_upper}
The Gaussian (sectional) curvature of $f^{*}\text{Re}H$ on $\check{D}^{+}$
has an upper bound.
\end{lemma}

Lemma \ref{lemma_gauss_upper} is proved in \cite[1.1.B]{gromov} under the assumption that
$M$ is compact. More details can be found in \cite[p.\ 219]{muller}. This argument certainly
works for us because $f(\check{D}^{+})$ is relatively compact.

\begin{lemma}\label{lemma_geodesic}
$\partial \check{D}^{+}$ is totally geodesic in $\check{D}^{+}$ with respect to $f^{*} \text{Re}H$.
\end{lemma}

As mentioned in \cite[1.3.C]{gromov}, Lemma \ref{lemma_geodesic} follows
from (1) and (2) in Lemma \ref{lemma_hermitian} and the fact that $f$ is $J$-holomorphic.
We omit a proof here since it's easy.

Define $\sigma: \mathbb{C} \rightarrow \mathbb{C}$ as the complex conjugation, i.e.\
$\sigma(z) = \bar{z}$. By (\ref{lower_half_disk}), we have
$\check{D}^{-} = \sigma (\check{D}^{+})$.
Define a Riemannian metric $G$ on $\check{D}$ as
\begin{equation}\label{metric_on_disk}
  G =
  \begin{cases}
     f^{*} \text{Re}H & \text{on $\check{D}^{+}$}, \\
     \sigma^{*} f^{*} \text{Re}H & \text{on $\check{D}^{-}$}.
  \end{cases}
\end{equation}

The importance of (2) in Lemma \ref{lemma_hermitian}, i.e.\ the fact that $W$ is totally geodesic, lies in
the following lemma.

\begin{lemma}\label{lemma_metric_on_disk}
The metric $G$ in (\ref{metric_on_disk}) is a well defined $C^{2}$
conformal metric on $\check{D}$.
In particular, the Gaussian curvature of $G$ makes sense.
Furthermore, $G$ is smooth on $\check{D}^{+}$ and $\check{D}^{-}$,
and $\sigma$ is an isometry.
\end{lemma}
\begin{proof}
By (1) in Lemma \ref{lemma_hermitian} and the fact that $f$ is $J$-holomorphic,
we see that $f^{*} \text{Re}H = \sigma^{*} f^{*} \text{Re}H$ on $\partial \check{D}^{+}
= \partial \check{D}^{-}$. Therefore, $G$ is well defined.

We also see that everything in this lemma is easily to be checked except maybe that
$G$ is $C^{2}$. Let's prove this.

Since $G$ is conformal, using the coordinate $z = x+iy$, we have
\[
G = g(x,y) dx \otimes dx + g(x,y) dy \otimes dy.
\]
We know that $g$ is continuous on $\check{D}$, and smooth on $\check{D}^{\pm}$.
Since $\sigma$ is an isometry, we also have
\begin{equation}\label{lemma_metric_on_disk_1}
g(x,y) = g(x,-y).
\end{equation}

By Lemma \ref{lemma_geodesic}, we infer that
\begin{equation}\label{lemma_metric_on_disk_2}
\frac{\partial g}{\partial y^{+}} (x,0) = 0,
\end{equation}
where $\frac{\partial g}{\partial y^{+}}$ is the upper half partial derivative
with respect to $y$.
By (\ref{lemma_metric_on_disk_1}) and (\ref{lemma_metric_on_disk_2}),
we infer that $\frac{\partial g}{\partial y} (x,0)$ exists. Then
$\frac{\partial^{2} g}{\partial y^{2}} (x,0)$ automatically exists because
(\ref{lemma_metric_on_disk_1}) tells us $g$ is an even function of $y$.

Now it's easy to check that $g$ is $C^{2}$ on $\check{D}$.
\end{proof}

Lemmas \ref{lemma_gauss_upper} and \ref{lemma_metric_on_disk} immediately
imply the following.

\begin{lemma}\label{lemma_gauss}
The Gaussian curvature of $G$ on $\check{D}$ has an upper bound.
\end{lemma}

Now we consider differential forms. We shall construct a bounded 1-form $\alpha'$ on $\check{D}$
which is a piecewise primitive of a symplectic form $\eta$ on $\check{D}$. Both forms $\alpha'$ and $\eta$
are obtained by a symmetric construction involving the pullbacks of $\alpha$ and $d\alpha$ to $\check{D}$.
The exactness of $\iota^*\alpha$ enables us to establish a good property
for $(\alpha',\eta)$ which ensures a Stokes' formula (see Example \ref{example_stokes_disk}).
Such a property will be useful in Lemma 4.9 to do certain arguments on
$\check{D}$ which are similar to arguments of Gromov used on $M$.

\begin{lemma}\label{lemma_form_zero}
There exists a $1$-form $\alpha_{1}$ on $M$ such that $\iota^{*} \alpha_{1} = 0$
on $W$ and $d \alpha_{1} = d \alpha$.
\end{lemma}
\begin{proof}
Since $\iota^{*} \alpha$ is exact, there exists a function $\mu$ on $W$
such that $\iota^{*} \alpha = d \mu$. Since $W$ is closed, we can extend
$\mu$ to be a function $\mu_{1}$ on $M$ such that $\iota^{*} \mu_{1} = \mu$.
We finish the proof by defining $\alpha_{1} = \alpha - d \mu_{1}$.
\end{proof}

Lemma \ref{lemma_form_zero} tells us that we may replace $\alpha$ in Theorem
\ref{theorem_tame} by $\alpha_{1}$. Therefore, from now on, we assume that
$\iota^{*} \alpha = 0$.

Define a $2$-form $\eta$ on $\check{D}$ as
\begin{equation}\label{2_form_disk}
  \eta =
  \begin{cases}
     f^{*} d \alpha & \text{on $\check{D}^{+}$}, \\
     -\sigma^{*} f^{*} d \alpha & \text{on $\check{D}^{-}$}.
  \end{cases}
\end{equation}
It's easy to check that $f^{*} d \alpha|_{\partial \check{D}^{+}}
= -\sigma^{*} f^{*} d \alpha|_{\partial \check{D}^{+}}$. Therefore,
$\eta$ is well defined. We infer that $\eta$ is smooth on $\check{D}^{+}$
and $\check{D}^{-}$, and continuous on $\check{D}$.

Similarly, define a $1$-form $\alpha'$ on $\check{D}$ as
\begin{equation}\label{1_form_disk}
  \alpha' =
  \begin{cases}
     f^{*} \alpha & \text{on $\check{D}^{+}$}, \\
     -\sigma^{*} f^{*} \alpha & \text{on $\check{D}^{-}$}.
  \end{cases}
\end{equation}
Let $\iota_{0}: \partial \check{D}^{+} \hookrightarrow \check{D}$
be the inclusion. By the assumptions that $\iota^{*} \alpha = 0$
and $f(\partial \check{D}^{+}) \subseteq W$, we get
\[
\iota_{0}^{*} f^{*} \alpha = 0.
\]
Therefore, it's easy to check that $\alpha'$ is well defined and
continuous on $\check{D}$.

It's important to observe that, if the assumption $\iota^{*} \alpha = 0$
is dropped, one cannot expect that $\alpha'$ is continuous on $\check{D}$.

Clearly, on $\check{D}^{+}$ and $\check{D}^{-}$, $\alpha'$ is smooth and
\begin{equation}\label{differential_disk}
d \alpha' = \eta.
\end{equation}

By Example \ref{example_no_extension}, Theorem \ref{theorem_tame} will no longer be
true if we drop the assumption that $\iota^{*} \alpha = 0$ (or more generally
$\iota^{*} \alpha$ is exact). How does this assumption help our proof? The following
is a quintessential example.

\begin{example}\label{example_stokes_disk}
Suppose $\Omega$ is a closed disk
in $\check{D}$. Then $\partial \check{D}^{+}$ divides $\Omega$
into two parts $\Omega^{+}$ and $\Omega^{-}$. (See Figure \ref{figure_1}.
The shadowed part is $\Omega$.)
\begin{figure}[!htbp]
\centering
\includegraphics[scale=0.24]{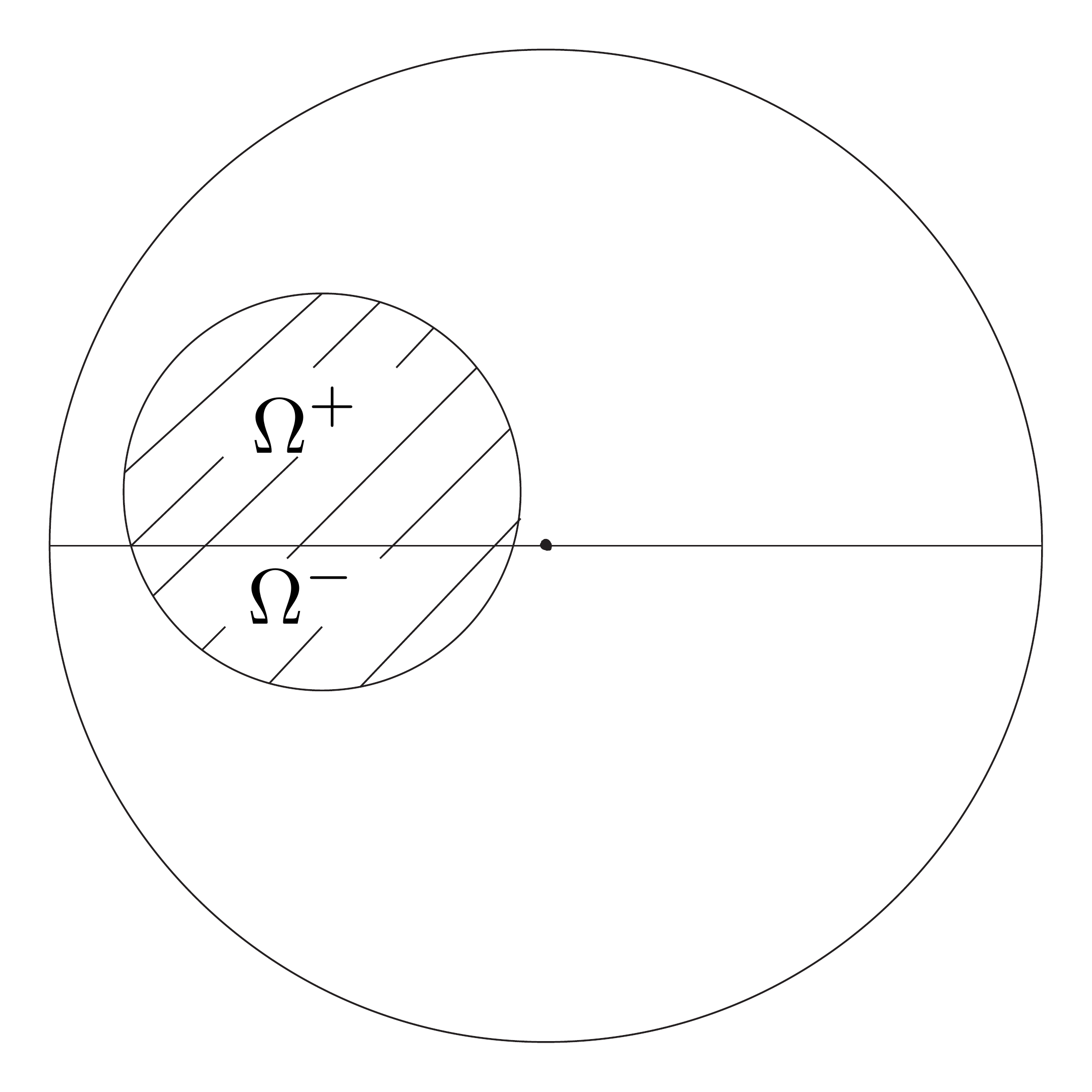} \caption{}
\label{figure_1}
\end{figure}

By (\ref{differential_disk}), we have
\[
\int_{\Omega^{\pm}} \eta = \int_{\Omega^{\pm}} d \alpha'
= \int_{\partial \Omega^{\pm}} \alpha'
\]
and
\[
\int_{\Omega} \eta = \int_{\Omega^{+}} \eta + \int_{\Omega^{-}} \eta
= \int_{\partial \Omega^{+}} \alpha' + \int_{\partial \Omega^{-}} \alpha'.
\]
Since $\iota_{0}^{*} \alpha' = 0$, we get the integrals of $\alpha'$ along
the real line is $0$. Therefore, we get the important formula
\begin{equation}\label{example_stokes_disk_1}
\int_{\Omega} \eta = \int_{\partial \Omega} \alpha'.
\end{equation}
Actually, as far as $\alpha'$ is continuous on $\check{D}$, we shall also get
(\ref{example_stokes_disk_1}) because the integrals on
$\partial \check{D}^{+}$ and $\partial \check{D}^{-}$ cancel each other.
However, as mentioned before, $\alpha'$ as defined above would not be continuous on $\check{D}$
if $\iota^{*}\alpha \neq 0$.
\end{example}

\begin{lemma}\label{lemma_form_disk}
There exist constants $C_{1}>0$ and $C_{2}>0$ such that
\[
G(v,v) \leq C_{1} \eta(v, iv)
\]
and
\[
\|\alpha'\|_{G} \leq C_{2},
\]
where $v$ is any tangent vector on $\check{D}$, $i$ is the complex structure on $\check{D}$ and $\|\alpha'\|_{G}$ is
the norm of $\alpha'$ with respect to $G$.
\end{lemma}
\begin{proof}
Since $f(\check{D}^{+})$ is relatively compact
and $d \alpha$ tames $J$, we infer there exist $C_{1}>0$ and $C_{2}>0$
such that, on $f(\check{D}^{+})$,
\[
H(w,w) \leq C_{1} d \alpha(w, Jw)
\]
and
\[
\|\alpha\|_{\text{Re}H} \leq C_{2},
\]
where $w$ is any tangent vector tangent to $f(\check{D}^{+})$.

By the definitions of $G$, $\eta$ and $\alpha'$, the lemma follows.
\end{proof}

Let $\varphi: D \rightarrow N$ be a $C^{1}$ $J$-holomorphic map,
where $N$ is a $C^{1}$ almost complex manifold. Let
$A(r)= |\varphi(D(r))|$ and $L(r) = |\partial \varphi(D(r))|$. Suppose
$A(r) \rightarrow 0$ when $r \rightarrow 0$. As we can
see in \cite[1.2 and 1.3]{gromov} and in this paper, studying the decaying
rate of $A(r)$ gives very important geometric information.

How to obtain such a decaying rate? Gromov tells us it's sufficient to use two tools.
The first one is the formula (\ref{gromov_formula}). The second one is an isoperimetric
inequality.

Following \cite[1.2 and 1.3]{gromov}, we shall build two types of isoperimetric
inequalities in this paper. The first one is $L^{2} \geq C A$ or more generally $L^{2} \geq C A
- C_{1}A^{2}$. The second one is $L^{2} \geq C A^{2}$. Comparing them,
the first one is better when $A$ is small, while the second one becomes better when
$A$ is large. The second one is important when we cannot bound $A(1)$, which lies at
the heart of how the form $\alpha$ plays a role in Theorem \ref{theorem_tame} and Lemma
\ref{lemma_tame_to_area} (see also Remark \ref{remark_isoperimetric_decay}).

For $z_{0} \in \check{D}$,
let $\varphi_{z_{0}}: D \rightarrow \check{D}$ be a holomorphic universal covering such that
$\varphi_{z_{0}}(0) = z_{0}$. It's well-known that $D$ carries a hyperbolic metric
$2 (1-|z|^{2})^{-1} |dz|$ and $\check{D}$ carries a hyperbolic metric
$\left( |z| \log \frac{1}{|z|} \right)^{-1} |dz|$, and
\begin{equation}\label{hyperbolic_covering}
\varphi_{z_{0}}: \left( D, \frac{2 |dz|}{1-|z|^{2}} \right) \rightarrow
\left( \check{D}, \frac{|dz|}{|z| \log \frac{1}{|z|}} \right)
\end{equation}
is a local isometry. Even though we use above the hyperbolic metrics to describe the map $\varphi_{z_0}$,
whenever not otherwise indicated, the metric on $D$
is the Euclidean metric, and the metric on $\check{D}$ is $G$. In particular the subsequent isoperimetric inequalities
are for the metric $G$ on $\check{D}$. (Compare also Definition \ref{definition_image_volume}).

By Lemma \ref{lemma_gauss} and the classical isoperimetric inequality in terms of
Gaussian curvature \cite[Theorem (1.2)]{barbosa_carmo} (see also \cite[p.\ 1206, (4.25)]{osserman}),
we obtain the following isoperimetric inequality.

\begin{lemma}\label{lemma_isoperimetric_gauss}
There exists a constant $C>0$ such that, for all $r \in (0,1)$ and $z_{0} \in \check{D}$,
\[
|\varphi_{z_{0}}(\partial D(r))|^{2} \geq 4 \pi \left( |\varphi_{z_{0}}(D(r))|
- C |\varphi_{z_{0}}(D(r))|^{2} \right).
\]
\end{lemma}

In the proof of the following isoperimetric inequality, we shall see the importance of the assumption
on $\alpha$ in Theorem \ref{theorem_tame}. (Compare \cite[p.\ 317, (8)]{gromov}.)

\begin{lemma}\label{lemma_isoperimetric_form}
There exists a constant $C>0$ such that, for all $r \in (0,1)$ and $z_{0} \in \check{D}$,
\[
|\varphi_{z_{0}}(\partial D(r))|^{2} \geq 2 \pi C |\varphi_{z_{0}}(D(r))|^{2}.
\]
\end{lemma}
\begin{proof}
We know that $\varphi_{z_{0}}^{-1}(\partial \check{D}^{+})$ is
a countable union of circular arcs whose ends are on the boundary of $D$.
(Figure \ref{figure_2} shows some of these circular arcs.
They are in fact geodesics of the hyperbolic metric on $D$.)

\begin{figure}[!htbp]
\centering
\includegraphics[scale=0.24]{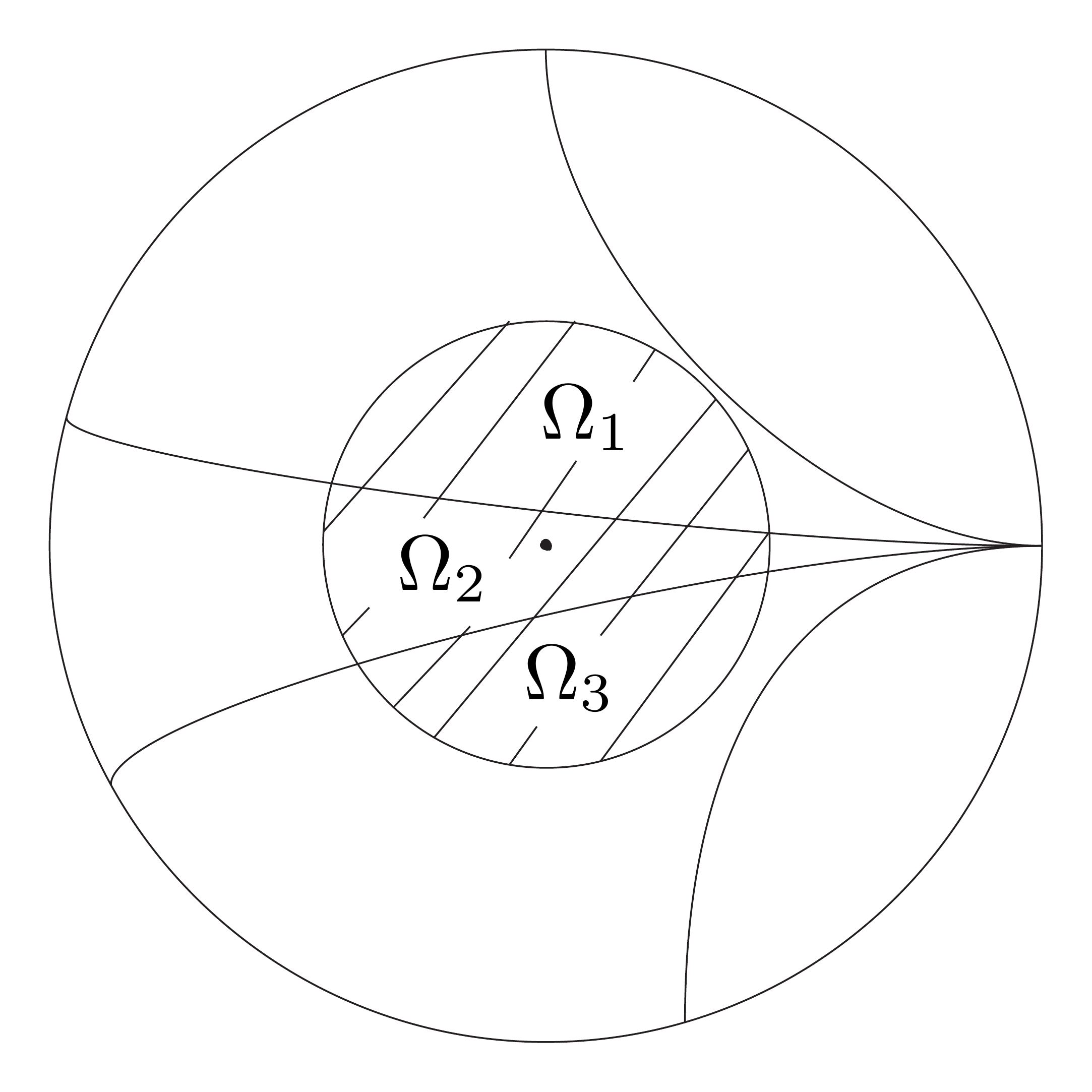} \caption{}
\label{figure_2}
\end{figure}

We consider the integral $\int_{D(r)} \varphi_{z_{0}}^{*} \eta$. Clearly,
$\varphi_{z_{0}}^{-1}(\partial \check{D}^{+})$ divides $D(r)$ into several domains
$\Omega_{1}, \cdots, \Omega_{k}$, which are compact submanifolds
with corners insider $D$. (This is illustrated by Figure \ref{figure_2}. The shadowed
part is $D(r)$.)
In each $\Omega_{i}$, by (\ref{differential_disk}), we have
\[
\varphi_{z_{0}}^{*} \eta = d \varphi_{z_{0}}^{*} \alpha'.
\]
Applying Stokes' formula, we get
\[
\int_{\Omega_{i}} \varphi_{z_{0}}^{*} \eta =
\int_{\partial \Omega_{i}} \varphi_{z_{0}}^{*} \alpha'.
\]
Similar to Example \ref{example_stokes_disk}, the line integral of $\varphi_{z_{0}}^{*} \alpha'$
vanishes on $\varphi_{z_{0}}^{-1}(\partial \check{D}^{+})$. Thus, summing up these integrals,
we get
\[
\int_{D(r)} \varphi_{z_{0}}^{*} \eta = \int_{\partial D(r)} \varphi_{z_{0}}^{*} \alpha'.
\]
By Lemma \ref{lemma_form_disk}, we get
\[
|\varphi_{z_{0}}(D(r))| \leq C_{1} \int_{D(r)} \varphi_{z_{0}}^{*} \eta
= C_{1} \int_{\partial D(r)} \varphi_{z_{0}}^{*} \alpha'
\leq C_{1} C_{2} |\varphi_{z_{0}}(\partial D(r))|,
\]
where $C_{1}$ and $C_{2}$ are the constants in Lemma \ref{lemma_form_disk}.
This finishes the proof.
\end{proof}

\begin{lemma}\label{lemma_bound_disk_covering}
There exists a constant $C>0$ such that
$\| d \varphi_{z_{0}} (0) \| \leq C$ for all maps $\varphi_{z_{0}}$.
\end{lemma}
\begin{proof}
Let $A(t)= |\varphi_{z_{0}}(D(t))|$ and $L(t) = |\varphi_{z_{0}}(\partial D(t))|$.
By (\ref{gromov_formula}) and Lemmas \ref{lemma_isoperimetric_form} and \ref{lemma_isoperimetric_gauss},
there exist constants $C_{3}>0$ and $C_{4}>0$ independent of $\varphi_{z_{0}}$ such that
\begin{equation}\label{lemma_bound_disk_covering_1}
\dot{A} \geq t^{-1} C_{3} A^{2},
\end{equation}
and
\begin{equation}\label{lemma_bound_disk_covering_2}
\dot{A} \geq 2 t^{-1} (A - C_{4} A^{2}).
\end{equation}

Since $A>0$ for $t \in (0,1)$, by (\ref{lemma_bound_disk_covering_1}), we infer
\[
A^{-2} \dot{A} \geq C_{3} t^{-1}.
\]
So, for $0 < r < \frac{1}{2}$, we have
\[
\int_{r}^{\frac{1}{2}} A^{-2} \dot{A} dt \geq
\int_{r}^{\frac{1}{2}} C_{3} t^{-1} dt
\]
or
\begin{equation}\label{lemma_bound_disk_covering_3}
A(r) \leq \frac{1}{A(\frac{1}{2})^{-1} - C_{3} \log 2r},
\end{equation}
where $A(\frac{1}{2}) < + \infty$.

Choose $r_{0}$ such that
\[
- C_{3} \log 2 r_{0} = 2 C_{4},
\]
we get
\begin{equation}\label{lemma_bound_disk_covering_4}
r_{0} = \frac{1}{2} e^{-\frac{2C_{4}}{C_{3}}}.
\end{equation}

By (\ref{lemma_bound_disk_covering_3}) and (\ref{lemma_bound_disk_covering_4}),
we get
\begin{equation}\label{lemma_bound_disk_covering_5}
A(r_{0}) \leq \frac{1}{A(\frac{1}{2})^{-1} - C_{3} \log 2r_{0}} \leq
\frac{1}{- C_{3} \log 2r_{0}} = \frac{1}{2 C_{4}}.
\end{equation}
For $0 < r \leq r_{0}$, since $A(r) \leq A(r_{0})$,  we have
\begin{equation}\label{lemma_bound_disk_covering_6}
A(r) - C_{4} A(r)^{2} \geq A(r)(1- C_{4} A(r_{0})) \geq 2^{-1} A(r) > 0.
\end{equation}

By (\ref{lemma_bound_disk_covering_2}) and (\ref{lemma_bound_disk_covering_6}),
we get, for $0 < r \leq r_{0}$,
\begin{equation}\label{lemma_bound_disk_covering_7}
(A - C_{4} A^{2})^{-1} \dot{A} \geq 2 t^{-1}.
\end{equation}
Integrating both sides of (\ref{lemma_bound_disk_covering_7}) on $[r,r_{0}]$, we get
\[
\log \frac{A(r_{0})}{1 - C_{4} A(r_{0})} - \log
\frac{A(r)}{1 - C_{4} A(r)}  \geq \log \left( \frac{r_{0}}{r} \right)^{2}
\]
or
\begin{equation}\label{lemma_bound_disk_covering_8}
\frac{A(r)}{1 - C_{4} A(r)} \leq
\frac{A(r_{0})}{1 - C_{4} A(r_{0})} \left( \frac{r}{r_{0}} \right)^{2}.
\end{equation}

By (\ref{lemma_bound_disk_covering_4}), (\ref{lemma_bound_disk_covering_5})
and (\ref{lemma_bound_disk_covering_8}), we get, for $r \leq r_{0}$,
\begin{eqnarray}\label{lemma_bound_disk_covering_9}
A(r) & \leq & \frac{A(r)}{1 - C_{4} A(r)} \leq
\frac{A(r_{0})}{1 - C_{4} A(r_{0})} \left( \frac{r}{r_{0}} \right)^{2}\\ \nonumber
& \leq & \frac{(2C_{4})^{-1}}{1 - C_{4} (2C_{4})^{-1}}
\left( \frac{1}{2} e^{-\frac{2C_{4}}{C_{3}}} \right)^{-2} r^{2}
= 4 C_{4}^{-1} e^{\frac{4C_{4}}{C_{3}}} r^{2}.
\end{eqnarray}

Therefore,
\[
\|d\varphi_{z_{0}}(0)\|^{2} = \lim_{r \rightarrow 0} (\pi r^{2})^{-1} A(r)
\leq 4 \pi^{-1} C_{4}^{-1} e^{\frac{4C_{4}}{C_{3}}},
\]
which finishes the proof.
\end{proof}

\begin{remark}\label{remark_isoperimetric_decay}
In the proof of Lemma \ref{lemma_bound_disk_covering}, by using Lemma \ref{lemma_isoperimetric_form},
we get the uniform estimate (\ref{lemma_bound_disk_covering_5}) for $A(r_{0})$, where both
(\ref{lemma_bound_disk_covering_5}) and $r_{0}$ are independent
of the maps $\varphi_{z_{0}}$.
Based on this estimate, Lemma \ref{lemma_isoperimetric_gauss} yields the decaying
rate (\ref{lemma_bound_disk_covering_9}).
\end{remark}

We are ready to conclude this section. Since the finiteness of $|f(\check{D}^{+}(r))|$ in Lemma \ref{lemma_tame_to_area}
follows from (\ref{lemma_tame_to_area_1}), it suffices to prove the following lemma which
implies (\ref{lemma_tame_to_area_1}) immediately.

\begin{lemma}\label{lemma_metric_hyperbolic_bound}
Using the standard coordinate $z=x+iy$ on $\check{D}$, the metric $G$ has the form
$g dx \otimes dx + g dy \otimes dy$. There exists a constant $C>0$ such that
\[
\sqrt{g(z)} \leq \frac{C}{|z| \log\frac{1}{|z|}}.
\]
\end{lemma}
\begin{proof}
Following the proof of \cite[1.3.A']{gromov}, the idea of this proof is to compare $G$
with the hyperbolic metric on $\check{D}$. (See also \cite[p.\ 227]{muller} and \cite[p.\ 42]{hummel}.)

Consider the above $\varphi_{z_{0}}: D \rightarrow \check{D}$ with
$\varphi_{z_{0}}(0) = z_{0}$. When $\varphi_{z_{0}}$ is viewed as a holomorphic function,
we denote its derivative by $\varphi_{z_{0}}'(z)$. Clearly,
\[
\| d\varphi_{z_{0}}(0) \| = \sqrt{g(z_{0})} |\varphi_{z_{0}}'(0)|.
\]

Since (\ref{hyperbolic_covering}) is a local isometry, we infer
\[
|\varphi_{z_{0}}'(0)| = 2 |z_{0}| \log \frac{1}{|z_{0}|} .
\]

By Lemma \ref{lemma_bound_disk_covering}, there exists a constant $C_{0}>0$
such that
\[
\| d\varphi_{z_{0}}(0) \| \leq C_{0}.
\]
Then we get
\[
\sqrt{g(z_{0})} \leq \frac{C_{0}}{2} \frac{1}{|z_{0}| \log \frac{1}{|z_{0}|}},
\]
which finishes the proof.
\end{proof}

\section{Pre-continuity}\label{section_pre-continuity}
The goal of this section is to prove the following lemma. It is the first step to
improve the regularity of $f$ at $0 \in D^{+}$ under the assumption of
Theorem \ref{theorem_finite_area}. The proof follows \cite[p.\ 135-136]{oh}
(compare \cite[p.\ 41-42]{hummel}) which is a combination of the Courant-Lebesgue Lemma
and a monotonicity lemma.

\begin{lemma}\label{lemma_converge_neighborhood}
Under the assumption of Theorem 2.6, suppose $U$ is a neighborhood of
$W$. Then there exists $r \in (0,1]$ such that
\[
f(\check{D}^{+}(r)) \subseteq U.
\]
\end{lemma}

In this section, we assume that the metric on $M$ is the $\text{Re}H$ in
Lemma \ref{lemma_hermitian} and $f$ is an embedding.

\begin{remark}
As one can check, the argument in this section does not use the manifold
structure on $W$.
If we only assume $W$ is a closed subset of $M$, Lemma
\ref{lemma_converge_neighborhood} will be still true.
\end{remark}

First, we formulate a monotonicity lemma which describes an important local property of a
$J$-holomorphic curve.

\begin{lemma}\label{lemma_monotone}
Let $K$ be a compact subset of $M$. There exists $\epsilon_{0} > 0$ and $C>0$
such that the following holds. Suppose $h: S \rightarrow M$ is an arbitrary immersed compact
$J$-holomorphic curve with boundary. Suppose $q \in h(S) \cap K$. Suppose
$r \in (0, \epsilon_{0})$ and $h(\partial S)$ is contained in the complement
of $B(q,r) \subseteq M$, where $B(q,r)$ is the open ball with center $q$ and
radius $r$. Then
\begin{equation}\label{lemma_monotone_1}
|h(S)| \geq C r^{2}.
\end{equation}
\end{lemma}

Lemma \ref{lemma_monotone} is a slight generalization of
\cite[p.\ 21, Theorem 1.3]{hummel}. (See also \cite[4.2.1]{muller}.)
It is proved in \cite[p.\ 26-28]{hummel} under the assumption that
$M$ is compact and $K=M$. The proof is based on an isoperimetric inequality
\cite[p.\ 23, Lemma 3.1]{hummel}. The book \cite[4.2.2]{muller} gives a quick proof
of such an isoperimetric inequality. Their argument actually works in our case.
Essentially, the compactness of $M$ is used in \cite{muller} and \cite{hummel}
for proving two facts. First, the injectivity radius of $M$ is positive. Second, $M$
is covered by finitely many open sets such that $J$ is tamed by an exact form in
each of these open sets. The second fact can be found in \cite[p.\ 224]{muller},
and will be addressed again in Section \ref{section_lipschitz_continuity} in
this paper (see Remark \ref{remark_monotone}). In our case,
$K$ also satisfies the above two properties because
of its compactness. Therefore, the argument in \cite{hummel} implies
Lemma \ref{lemma_monotone}. We shall omit a proof here.

Now we prove a Courant-Lebesgue Lemma for our map $f: \check{D} \rightarrow M$.
A much more general result can be found in Courant's book
\cite[p.\ 101, Lemma\ 3.1]{courant} and \cite[p.\ 239]{d_h_k_w}. Our proof
follows \cite{courant}. (See also \cite[Lemma 4.2]{oh}.)

For $t \in (0,1)$, define
\[
\Gamma_{t} = \{ z \in \check{D}^{+} \mid |z|=t \}.
\]

\begin{lemma}[Courant-Lebesgue Lemma]\label{lemma_courant_lebesgue}
For all $r \in (0,1)$, there exists $t_{0} \in [r^{2},r]$ such that
\[
|f(\Gamma_{t_{0}})|^{2} \leq \frac{\pi |f(\check{D}^{+})|}{\log \frac{1}{r}}.
\]
\end{lemma}
\begin{proof}
\begin{eqnarray*}
|f(\check{D}^{+})| & \geq & \int_{\check{D}^{+}(r) - \check{D}^{+}(r^{2})}
\| df \|^{2} \\
& = & \int_{r^{2}}^{r} \left( \int_{\Gamma_{t}} \|df\|^{2} \right) dt \\
& \geq & \int_{r^{2}}^{r} \left( \int_{\Gamma_{t}} \|df\| \right)^{2}
\left( \int_{\Gamma_{t}} 1 \right)^{-1}  dt \\
& = & \int_{r^{2}}^{r} |f(\Gamma_{t})|^{2}
(\pi t)^{-1}  dt.
\end{eqnarray*}
Since $|f(\Gamma_{t})|$ is continuous with respect to $t \in [r^{2},r]$, there exists
$t_{0} \in [r^{2},r]$ such that $|f(\Gamma_{t_{0}})|$ attains the minimum.
Therefore,
\[
|f(\check{D}^{+})| \geq (\pi)^{-1} |f(\Gamma_{t_{0}})|^{2}
\int_{r^{2}}^{r} t^{-1} dt = (\pi)^{-1} \left( \log \frac{1}{r} \right)
|f(\Gamma_{t_{0}})|^{2},
\]
which completes the proof.
\end{proof}

\begin{remark}
More generally, if $f$ is not $J$-holomorphic, we still have a Courant-Lebesgue
Lemma. However, we have to replace the $|f(\check{D}^{+})|$
in Lemma \ref{lemma_courant_lebesgue} by the energy
(or the Dirichlet integral) of $f$. See \cite{courant}.
\end{remark}

\begin{proof}[Proof of Lemma \ref{lemma_converge_neighborhood}]
Let $\overline{f(\check{D}^{+})}$ be the closure of $f(\check{D}^{+})$.
Then $\overline{f(\check{D}^{+})}$ is compact. By Lemma \ref{lemma_monotone},
there exists $\epsilon_{0} > 0$ such that the conclusion of Lemma
\ref{lemma_monotone} holds for $K = \overline{f(\check{D}^{+})}$.

Since $|f(\check{D}^{+})| < + \infty$, by Lemma \ref{lemma_courant_lebesgue},
there exists a decreasing sequence $\{r_{n}\}$ such that
$r_{n} \rightarrow 0$ and $|f(\Gamma_{r_{n}})| \rightarrow 0$
when $n \rightarrow \infty$. Define
\[
\Omega_{n} = \{ z \in \check{D}^{+} \mid r_{n+1} \leq |z| \leq r_{n} \}.
\]

We prove this lemma by contradiction. We may assume $U$ is an open subset.

Suppose this lemma is not true. Then, choosing a subsequence of $\{r_{n}\}$
if necessary, there exist $z_{n} \in \Omega_{n}$ such that $f(z_{n}) \notin U$.

Denote $f(z_{n})$ by $p_{n}$, we have $p_{n} \in \overline{f(\check{D}^{+})}
- U$. Since $\overline{f(\check{D}^{+})}$ is compact and $U$ is open, we
infer $\overline{f(\check{D}^{+})} - U$ is compact. Choosing a subsequence
of $\{r_{n}\}$ if necessary, we get a sequence $\{p_{n}\}$ such that
$p_{n} \in f(\Omega_{n})$ and $p_{n} \rightarrow p \notin U$ when $n \rightarrow \infty$.

Since $W$ is closed and $W \subseteq U$, we infer that the
distance $d_{0} = d(p,W) > 0$. Since the boundary of $f(\Gamma_{r_{n}})$
is in $W$ and $|f(\Gamma_{r_{n}})| \rightarrow 0$, we get
\[
\max_{q \in f(\Gamma_{r_{n}})} d(q,W) \rightarrow 0,
\]
when $n \rightarrow \infty$. Clearly,
\[
f(\partial \Omega_{n}) \subseteq W \cup f(\Gamma_{r_{n}}) \cup f(\Gamma_{r_{n+1}}).
\]
Thus, for $n$ large enough, we have
\[
d(p, f(\partial \Omega_{n})) > \frac{2d_{0}}{3}.
\]

Deleting first finitely many terms of $\{p_{n}\}$ if necessary, we have $p_{n} \in B(p, \frac{d_{0}}{3})$
for all $n$. We infer that
\begin{equation}\label{lemma_converge_neighborhood_1}
d(p_{n}, f(\partial \Omega_{n})) > \frac{d_{0}}{3}.
\end{equation}

Choose $\epsilon_{1} > 0$ such that $\epsilon_{1} < \min \{ \epsilon_{0}, \frac{d_{0}}{3} \}$.
Since $\epsilon_{1} < \frac{d_{0}}{3}$, by (\ref{lemma_converge_neighborhood_1}),
there exists $S_{n} \subseteq \Omega_{n}$ such that $p_{n} \in S_{n}$ and
$f(\partial S_{n})$ is outside of $B(p, \epsilon_{1})$, where $S_{n}$ is a compact
surface with boundary. Since $\epsilon_{1} < \epsilon_{0}$, taking $h=f$,
$S = S_{n}$ and $r = \epsilon_{1}$ in (\ref{lemma_monotone_1}),
we get
\[
|f(S_{n})| \geq C \epsilon_{1}^{2}.
\]

Thus,
\[
|f(\check{D}^{+})| \geq \sum_{n=1}^{\infty} |f(\Omega_{n})|
\geq \sum_{n=1}^{\infty} |f(S_{n})| \geq \sum_{n=1}^{\infty}
C \epsilon_{1}^{2} = + \infty,
\]
which contradicts the assumption that $|f(\check{D}^{+})| < + \infty$.
\end{proof}

\section{Continuity}\label{section_continuity}
The main goal of this section is to prove that, under the assumption of Theorem
\ref{theorem_finite_area}, $f$ has a continuous extension over $0 \in D^{+}$.
This will follow from a stronger Lemma \ref{lemma_continuity} which gives a continuous
extension of the extrinsic doubling map $F$.

The extrinsic doubling map $F$ pulls back important information from $M$ to $\check{D}$, which makes
our intrinsic doubling more efficient. In particular, this helps us take advantage of
the ``almost minimal" property of $J$-holomorphic curves (see comment before Lemma
\ref{lemma_isoperimetric_disk}).

First of all, the combination of the intrinsic doubling and Lemmas \ref{lemma_hermitian}
and \ref{lemma_converge_neighborhood} already gives us some information.

We assume $M$ is equipped with $\text{Re}H$ and $f$ is an embedding.

By using Lemma \ref{lemma_hermitian} and the exponential map for the bundle $J(TW) \rightarrow W$,
there exists an open neighborhood $U$ of $W$
such that the following holds: (1).\ $U$ is identified with a neighborhood of
the zero section of the bundle $J(TW) \rightarrow W$. (2).\ The $1$-form $\alpha_{0}$
in (3) of Lemma \ref{lemma_hermitian} is defined in $U$.

Consider the natural reflection
\begin{eqnarray}\label{reflection}
\tau: J(TW) & \rightarrow & J(TW), \\
\tau(p,v) & = & (p,-v), \nonumber
\end{eqnarray}
where $p$ is a base point on $W$ and $v \in J(T_{p}W)$. Then $\tau^{2} = \text{Id}$
and $\tau$ is a diffeomorphism with the fixed point set $W$. Shrinking $U$ if
necessary, we obtain
\[
\tau(U) = U.
\]

Define a new almost complex structure $\widetilde{J}$ on $U$ as
\begin{equation}\label{new_complex_structure}
\widetilde{J} = - d \tau J d \tau.
\end{equation}
Then $U$ has two almost complex structures $J$ and $\widetilde{J}$ and
$\tau: (U, J) \rightarrow (U, \widetilde{J})$
is an anti-$J$-holomorphic isomorphism, i.e.\ $d \tau J =
- \widetilde{J} d \tau$. For all $p \in W$, $v_{1} \in T_{p} W
\subseteq T_{p} U$ and $v_{2} \in J(T_{p} W)
\subseteq T_{p} U$, we have
\[
d \tau v_{1} = v_{1}, \qquad \text{and} \qquad d \tau v_{2} = - v_{2}.
\]
Therefore, $\widetilde{J}$ equals $J$ on $TU|_{W}$. Then choose $U$ sufficiently
small, we get the following.

\begin{lemma}\label{lemma_tame_both}
$d \alpha_{0}$ tames both $J$ and $\widetilde{J}$.
\end{lemma}

Clearly, $\tau^{*}(\text{Re}H)$ is also a Riemannian metric on $U$.
By (1) of Lemma \ref{lemma_hermitian},
we infer that $\tau^{*}(\text{Re}H)$ equals $\text{Re}H$ on $TU|_{W}$.
It's easy to see that $\widetilde{J}$ preserves $\tau^{*}(\text{Re}H)$.
Thus we can construct a Hermitian $\widetilde{H}$ with respect to $\widetilde{J}$ on $U$ such that
\begin{equation}
\text{Re}\widetilde{H} = \tau^{*}(\text{Re}H).
\end{equation}
Then we get the following.

\begin{lemma}\label{lemma_structrue_in_u}
$\widetilde{H}$ is a Hermitian with respect to $\widetilde{J}$. We have
$J = \widetilde{J}$ and $H=\widetilde{H}$ on $TU|_{W}$.
\end{lemma}

\begin{remark}
We would like to point out that $\widetilde{H} \neq \tau^{*} H$ because $\tau$
is \textit{anti-$J$-holomorphic}. Actually, we have
$\widetilde{H}(v_{1},v_{2}) = \overline{\tau^{*}H(v_{1},v_{2})}$,
where $\overline{\tau^{*}H(v_{1},v_{2})}$ is the complex conjugate of
$\tau^{*}H(v_{1},v_{2})$.
\end{remark}

Since $W$ is a closed subset of $M$ and $M$ is a metric space, there exists a neighborhood $U_{1}$ of $W$
such that the closure of $U_{1}$ is contained in $U$.
By Lemma \ref{lemma_converge_neighborhood}, there exists
$0 < r \leq 1$ such that $f(\check{D}^{+}(r)) \subseteq U_{1}$.
Since $f(\check{D}^{+})$ is relatively compact in $M$, we infer that
$f(\check{D}^{+}(r))$ is relatively compact in $U$.
Therefore, without loss of generality, we may assume
\[
f(\check{D}^{+}) \subseteq U,
\]
and it is relatively compact in $U$.

We need the metric $G$ on $\check{D}$ defined in (\ref{metric_on_disk})
again. Since $d \alpha_{0}$ tames $J$ in $U$, we know $f$ satisfies
the assumption of Theorem \ref{theorem_tame}. Therefore, all facts
about $G$ in Section \ref{section_reduction_lemma} can be used now.
By Lemma \ref{lemma_metric_hyperbolic_bound}, we get the following.

\begin{lemma}\label{lemma_circle_zero}
The length of $\partial D(r)$ with respect to $G$ converges to $0$
when $r$ converges to $0$.
\end{lemma}

Let $\sigma$ be the complex conjugation on $\check{D}$ as in Section \ref{section_reduction_lemma}.
We double $f: \check{D}^{+} \rightarrow U$ to be $F: \check{D} \rightarrow U$, where
\begin{equation}
F =
\begin{cases}
f & \text{on $\check{D}^{+}$},\\
\tau f \sigma & \text{on $\check{D}^{-}$}.
\end{cases}
\end{equation}

As we shall see later, Lemma \ref{lemma_circle_zero} implies that the length of
$F(\partial D(r))$ shrinks to $0$ when $r$ goes to $0$. In order to obtain the continuous
extension of $F$, following Gromov's idea in the interior case, it suffices to show that
the diameter of $F(\check{D}(r))$ is bounded in terms of
the length of $F(\partial D(r))$. This is the main task
in the section.

\begin{lemma}
$F$ is a well defined $C^{1}$ immersion. It is smooth on
$\check{D}^{+}$ and $\check{D}^{-}$. And $F(\check{D})$
is relatively compact.
\end{lemma}
\begin{proof}
Since $\partial \check{D}^{+}$ and $W$ are the fixed point sets of $\sigma$
and $\tau$ respectively, and $f(\partial \check{D}^{+}) \subseteq W$,
we have $f|_{\partial \check{D}^{+}} = \tau f \sigma|_{\partial \check{D}^{+}}$.
Therefore, $F$ is well defined and continuous.
Clearly, $F$ is smooth on $\check{D}^{+}$ and $\check{D}^{-}$.

Consider the coordinate $z=x+iy$ on $\check{D}$. Since $f$ is $J$-holomorphic,
and $f(\partial \check{D}^{+}) \subseteq W$, we infer that, for
$z \in \partial \check{D}^{+}$,
\[
df(z) \frac{\partial}{\partial y} \in J(T_{f(z)} W).
\]
By (\ref{reflection}), we get
\[
d (\tau f \sigma)(z) \frac{\partial}{\partial y} = df(z) \frac{\partial}{\partial y}.
\]
Then it's easy to check that $F$ is $C^{1}$. Since $f$ is an embedding, $F$ is an immersion.

Finally, since $F(\check{D}) = f(\check{D}^{+}) \cup \tau f(\check{D}^{+})$
and $f(\check{D}^{+})$ is relatively compact, we have $F(\check{D})$ is also
relatively compact.
\end{proof}

\begin{remark}\label{remark_c_1}
There is some inconvenience caused by the modest $C^{1}$ regularity of $F$.
For example, if $\alpha$ is a smooth form on $U$, then $F^{*} \alpha$ is only a $C^{0}$
form. We have trouble in applying Stokes' formula. But this can always be saved
by the piecewise smoothness of $F$.
\end{remark}

A natural question is if $F: \check{D} \rightarrow (U,J)$ or
$F: \check{D} \rightarrow (U,\widetilde{J})$ is $J$-holomorphic. Unfortunately,
we cannot give an affirmative answer. However, fortunately enough, we still have the above
Lemma \ref{lemma_structrue_in_u} and the following lemma.

\begin{lemma}\label{lemma_structure_F}
$F: \check{D}^{+} \rightarrow (U,J)$ and $F: \check{D}^{-} \rightarrow (U,\widetilde{J})$
are $J$-holomorphic. Furthermore,
\[
G|_{\check{D}^{+}} = F^{*} \text{Re}H \qquad \text{and} \qquad
G|_{\check{D}^{-}} = F^{*} \text{Re}\widetilde{H}.
\]
\end{lemma}

The following are two special examples to illustrate the above construction.

\begin{example}\label{example_schwarz}
Take $M = \mathbb{C}$ and $W = \mathbb{R}$. Then $f$ is a holomorphic function
which takes real values on $\partial \check{D}^{+}$. Define the above $\tau$
as the complex conjugation. Then, on $\check{D}^{-}$,
we have $F(z) = \overline{f(\overline{z})}$. By the Schwarz reflection
principle, $F$ is holomorphic on $\check{D}$.
\end{example}

\begin{example}\label{example_pansu}
More general than Example \ref{example_schwarz}, suppose there exists a neighborhood
$U_{0}$ of $W$ such that the following holds. The almost complex structure is integrable
in $U_{0}$, and $W$ is real analytic with respect to this complex structure.
Then \cite[p.\ 244-245]{pansu} (see also \cite[2.1.D]{gromov}) defines a reflection
$\tau_{M}$ in $U_{0}$ slightly different from (\ref{reflection}). Choose complex coordinate
charts such that the coordinates take pieces of $W$ into $\mathbb{R}^{n} \subseteq \mathbb{C}^{n}$.
Define $\tau_{M}$ as the complex conjugation in these charts. This definition globally makes
sense because the transition functions of these charts are holomorphic.
Suppose $f(\check{D}^{+}) \subseteq U_{0}$. Similar to Example \ref{example_schwarz},
a holomorphic doubling map $F: \check{D} \rightarrow U_{0}$ is obtained.
\end{example}

In the above examples, one gets a $J$-holomorphic doubling $F$ because $(M,W,J)$ has sufficient
symmetry. Our case is a further extension of Example \ref{example_pansu}. We have
difficulty to get a $J$-holomorphic $F$ unless $J=\widetilde{J}$.
However, we can consider the pair $(J,\widetilde{J})$ together. The idea comes from the following phenomenon.
Consider a group action on a space. If $X$ is a fixed point, then it has sufficient symmetry
such that a symmetric argument can be applied on it. On the other hand,
if $X$ is not a fixed point, then
we can still apply such an argument on its orbit.

It's more convenient to consider $(D,G)$ than $(\check{D},G)$. More precisely, we shall
consider $G$ as a metric on $D$ with a singularity $0$ where $G$ is undefined.
For clarity, we shall use $|\cdot|_{G}$ to denote the length or area with respect to $G$.

Define $\Omega$ as a closed disk inside $D$,
\begin{equation}\label{disk_in_disk}
\Omega = \{z \in D \mid |z-z_{0}| \leq r, z_{0} \in D, |z_{0}| + r < 1 \}.
\end{equation}

We shall consider the integral $\int_{\Omega} F^{*} d \alpha_{0}$.
As $d \alpha_{0}$ tames both $J$ and $\widetilde{J}$, this integral is actually
an integral of a positive function on $\Omega-\{0\}$. So it makes sense.

\begin{lemma}\label{lemma_form_integral}
There exists a constant $C>0$ independent of $\Omega$ such that
\begin{equation}\label{lemma_form_integral_1}
|\Omega|_{G} \leq C \int_{\Omega} F^{*} d \alpha_{0},
\end{equation}
where $|\Omega|_{G}$ is the area of $\Omega - \{0\}$ with respect to $G$.
Furthermore, if $0 \notin \partial \Omega$, then
\begin{equation}\label{lemma_form_integral_2}
\int_{\Omega} F^{*} d \alpha_{0} = \int_{\partial \Omega} F^{*} \alpha_{0}.
\end{equation}
\end{lemma}
\begin{proof}
By Lemma \ref{lemma_structure_F}, we get
\[
|\Omega|_{G}
= |\Omega \cap \check{D}^{+}|_{G} + |\Omega \cap \check{D}^{-}|_{G}
= \int_{\Omega \cap \check{D}^{+}} F^{*} (-\text{Im}H) +
\int_{\Omega \cap \check{D}^{-}} F^{*} (-\text{Im}\widetilde{H}).
\]
Since $F(\check{D})$ is relatively compact, by Lemma \ref{lemma_tame_both},
there exists a constant $C>0$ such that
$- \text{Im}H(v,Jv) \leq C d \alpha_{0}(v,Jv)$ and
$- \text{Im}\widetilde{H}(v,\widetilde{J}v) \leq C d \alpha_{0}(v,\widetilde{J}v)$
on $F(\check{D})$. Then we obtain
\[
\int_{\Omega \cap \check{D}^{+}} F^{*} (-\text{Im}H) +
\int_{\Omega \cap \check{D}^{-}} F^{*} (-\text{Im}\widetilde{H})
\leq C \int_{\Omega} F^{*} d \alpha_{0}.
\]
Since $C$ only depends on $F(\check{D})$, we proved (\ref{lemma_form_integral_1}).

For (\ref{lemma_form_integral_2}), we only prove the case of that $0 \in \Omega$.
The proof for the case of that $0 \notin \Omega$ is similar and even easier.

Since $0 \notin \partial \Omega$, we have
$D(\epsilon) = \{ z \mid |z| < \epsilon \} \subseteq \Omega$ when $\epsilon >0$ is
small enough. Consider the domain
\[
\Omega_{\epsilon} = \Omega - D(\epsilon).
\]
Define $\Omega_{\epsilon}^{+} = \Omega_{\epsilon} \cap \check{D}^{+}$
and $\Omega_{\epsilon}^{-} = \Omega_{\epsilon} \cap \check{D}^{-}$.
Since $F$ is smooth on $\Omega_{\epsilon}^{\pm}$, and $\Omega_{\epsilon}^{\pm}$
is a compact manifold with corners, applying Stokes' formula, we get
\[
\int_{\Omega_{\epsilon}^{\pm}} F^{*} d \alpha_{0}
= \int_{\Omega_{\epsilon}^{\pm}}  d F^{*} \alpha_{0}
= \int_{\partial \Omega_{\epsilon}^{\pm}}  F^{*} \alpha_{0}.
\]
Since $F$ is $C^{1}$ on $\check{D}$, we infer that $F^{*} \alpha_{0}$ is
continuous on $\check{D}$. Therefore, the integrals of $F^{*} \alpha_{0}$ on the
real line cancel each other, i.e.
\[
\int_{\Omega_{\epsilon}} F^{*} d \alpha_{0}
= \int_{\partial \Omega_{\epsilon}}  F^{*} \alpha_{0}
= \int_{\partial \Omega}  F^{*} \alpha_{0}
- \int_{\partial D(\epsilon)}  F^{*} \alpha_{0}.
\]

Since $F(\check{D})$ is relatively compact, we infer
$\| \alpha_{0} \|_{\text{Re}H}$ and $\| \alpha_{0} \|_{\text{Re}\widetilde{H}}$
are bounded on $F(\check{D})$.
By Lemma \ref{lemma_circle_zero}, $|\partial D(\epsilon)|_{G} \rightarrow 0$
when $\epsilon \rightarrow 0$. Thus $\int_{\partial D(\epsilon)} F^{*} \alpha_{0}
\rightarrow 0$. Furthermore,
$\int_{\Omega} F^{*} d \alpha_{0}$ is actually an integral of a positive
function on $\Omega-\{0\}$. Since the integral
$\int_{\Omega_{\epsilon}} F^{*} d \alpha_{0}$ is monotone increasing
when $\epsilon \rightarrow 0$, we get
\[
\int_{\Omega} F^{*} d \alpha_{0}
= \lim_{\epsilon \rightarrow 0} \int_{\Omega_{\epsilon}} F^{*} d \alpha_{0}
= \int_{\partial \Omega}  F^{*} \alpha_{0}.
\]
\end{proof}

In \cite[1.3.B \& 1.3.B']{gromov}, Gromov makes the following crucial observation:
A $J$-holomorphic curve is ``almost minimal", which results in isoperimetric
inequalities. The idea is as follows. Suppose $J$ is tamed by an exact form.
(This is true in our case by Lemma \ref{lemma_tame_both}.) Let $S$ be a compact $J$-holomorphic
curve with boundary. Then the area of $S$ is controlled
by some other compact surfaces $K$
with $\partial K = \partial S$. If a special $K$ satisfies an isoperimetric
inequality, then we can in turn establish an isoperimetric inequality for $S$.
This is indeed the case if $\partial S$ lies in a suitable coordinate chart. Gromov's
choice of $K$ is a minimal surface with respect to the Euclidean metric of
the chart. However, as shown in \cite[p.\ 25]{hummel}, there is an even
more elementary choice. That is the cone constructed at the end of Section \ref{section_set_up}.

Applying Gromov's idea, we obtain the following isoperimetric inequality which is
the most important for this section.

\begin{lemma}\label{lemma_isoperimetric_disk}
There exists a constant $\mu >0$ such that the following holds. For all disks
$\Omega$ defined in (\ref{disk_in_disk}) such that $0 \notin \partial \Omega$,
we have
\[
|\partial \Omega|_{G}^{2} \geq 2 \pi \mu |\Omega|_{G}.
\]
\end{lemma}
\begin{proof}
Since $F(\check{D})$ is relatively compact, $F(\check{D})$ can be covered by finitely
many subsets $B_{j}$ ($j=1, \dots, k$) of $U$ which have
the following properties: (1).\ Each $B_{j}$ is in a coordinate chart. (2).\ By
these coordinates, each $B_{j}$ is identified with
a closed ball in $\mathbb{R}^{2n}$.
(3).\ In each $B_{j}$, $\text{Re}H$ and $\text{Re}\widetilde{H}$ are equivalent with
the Euclidean metric on $\mathbb{R}^{2n}$, and $d \alpha_{0}$ is bounded by the
Euclidean metric. (4).\ There exists a constant $C_{1} > 0$  such
that, if $\Gamma$ is a curve in $\check{D}$ with $|\Gamma|_{G} \leq C_{1}$, then $F(\Gamma)$ is contained
in one of these $B_{j}$.

If $|\partial \Omega|_{G} > C_{1}$, we get
\[
|\partial \Omega|_{G}^{2}
=\frac{|\partial \Omega|_{G}^{2}}{|\Omega|_{G}} |\Omega|_{G}
> \frac{C_{1}^{2}}{|\check{D}|_{G}} |\Omega|_{G}.
\]
Since $|\check{D}|_{G} = 2 |\check{D}^{+}|_{G} < + \infty$,
we already obtained the isoperimetric inequality. Therefore, it remains to check the case
of that $|\partial \Omega|_{G} \leq C_{1}$.

By the above property (4), we may assume $F(\partial \Omega) \subseteq B_{1}$.
To make the idea clear, we first prove the case of that either $\Omega \subseteq D^{+}$
or $\Omega \subseteq D^{-}$.

Since $0 \notin \partial \Omega$, by Lemma \ref{lemma_form_integral}, we get
\begin{equation}\label{lemma_isoperimetric_disk_1}
|\Omega|_{G} \leq C \int_{\partial \Omega} F^{*} \alpha_{0}
= C \int_{F(\partial \Omega)} \alpha_{0}.
\end{equation}

Let $\gamma: S^{1} \rightarrow F(\partial \Omega)$ be a reparametrization of
$F(\partial \Omega)$. Since $\Omega \subseteq D^{+}$ or
$\Omega \subseteq D^{-}$, and $0 \notin \partial \Omega$, we infer that $F$ is smooth
on $\partial \Omega$. Thus we can also require that $\gamma$ is smooth. Since $F$ is an
immersion, we can also require that $\gamma$ is an immersion.
As Section \ref{section_set_up}, using the Euclidean metric $\| \cdot \|_{\text{Eu}}$
of this chart, we construct
the cone $K(\gamma(S^{1}))$ with vertex $c$ in (\ref{mass_center}) and the boundary
$\gamma(S^{1}) = F(\partial \Omega)$.

Define a smooth function $\rho: [0,1] \rightarrow [0,1]$ such
that $\rho|_{[0,\frac{1}{2}]}=0$ and $\rho$ strictly increases from $0$ to $1$ on
$[\frac{1}{2}, 1]$. Denote the closed unit disk by $\overline{D}$.
Use the angle coordinate $\theta$ for $S^{1}$.
Define $\xi: \overline{D} \rightarrow B_{1}$
as
\begin{equation}\label{lemma_isoperimetric_disk_2}
  \xi(r e^{i \theta}) =
  \begin{cases}
     c + \rho(r) (\gamma(\theta) - c) & r \geq \frac{1}{2}, \\
     c & r \leq \frac{1}{2}.
  \end{cases}
\end{equation}
Then $\xi(\overline{D})$ is $K(\gamma(S^{1}))$.
Since $B_{1}$ is convex and closed in $\mathbb{R}^{2n}$,
by (\ref{mass_center}), $c$ and the cone $K(\gamma(S^{1}))$ are in $B_{1}$.
Therefore, the map $\xi$ is well defined.

By Lemma \ref{lemma_cone_isoperimetric}, we get
\begin{equation}\label{lemma_isoperimetric_disk_3}
|\xi(\overline{D})|_{\text{Eu}} \leq (4 \pi)^{-1} |\gamma(S^{1})|_{\text{Eu}}^{2}.
\end{equation}
Since $\gamma$ and $\rho$ are smooth, we know that $\xi$ is smooth. Therefore,
\begin{equation}\label{lemma_isoperimetric_disk_4}
\int_{F(\partial \Omega)} \alpha_{0}
= \int_{\xi(\partial \overline{D})} \alpha_{0}
= \int_{\xi(\overline{D})} d \alpha_{0}.
\end{equation}
Since $d \alpha_{0}$ is bounded by the Euclidean metric in $B_{1}$, there
exists $C_{2} > 0$ such that
\begin{equation}\label{lemma_isoperimetric_disk_5}
\int_{\xi(\overline{D})} d \alpha_{0} \leq C_{2} |\xi(\overline{D})|_{\text{Eu}}.
\end{equation}
Furthermore, since $\text{Re}H$ and $\text{Re}\widetilde{H}$ are equivalent to
the Euclidean metric in $B_{1}$, by Lemma \ref{lemma_structure_F},
we infer that there exists $C_{3}>0$ such that
\begin{equation}\label{lemma_isoperimetric_disk_6}
|\gamma(S^{1})|_{\text{Eu}}^{2} \leq C_{3} |\partial \Omega|_{G}^{2}.
\end{equation}

By (\ref{lemma_isoperimetric_disk_1}), (\ref{lemma_isoperimetric_disk_4}),
(\ref{lemma_isoperimetric_disk_5}), (\ref{lemma_isoperimetric_disk_3})
and (\ref{lemma_isoperimetric_disk_6}), we get
\[
|\Omega|_{G} \leq C \int_{F(\partial \Omega)} \alpha_{0}
= C \int_{\xi(\overline{D})} d \alpha_{0}
\leq C C_{2} (4\pi)^{-1} C_{3} |\partial \Omega|_{G}^{2},
\]
where the constant $C C_{2} (4\pi)^{-1} C_{3}$ only depends on $B_{1}$.
Since there are only finite many $B_{j}$, we infer that there
exists $C_{4} > 0$ independent of $\Omega$ such that
\begin{equation}\label{lemma_isoperimetric_disk_7}
|\partial \Omega|_{G}^{2} \geq C_{4 } |\Omega|_{G}.
\end{equation}

Now we deal with the case that $|\partial \Omega|_{G} \leq C_{1}$
and $\Omega$ is contained neither $D^{+}$ nor $D^{-}$. We shall prove that
(\ref{lemma_isoperimetric_disk_7}) still
holds. The proof is the same as the one of the previous case with one exception:
the maps $F|_{\partial \Omega}$ and $\xi$ in (\ref{lemma_isoperimetric_disk_2})
are only $C^{1}$ now. (See Remark \ref{remark_c_1}.)

Checking the above argument, it suffices
to show that (\ref{lemma_isoperimetric_disk_4}) is still true.

Since $F$ is smooth on $\check{D}^{\pm}$ and $C^{1}$ on $\check{D}$, we can require
that the reparametrization $\gamma: S^{1} \rightarrow F(\partial \Omega)$
satisfies the following properties: $\gamma$ is $C^{1}$, $\gamma([0, \theta_{0}])
\subseteq F(\check{D}^{+})$, $\gamma([\theta_{0}, 2 \pi]) \subseteq F(\check{D}^{-})$,
and $\gamma$ is smooth on $[0, \theta_{0}]$ and $[\theta_{0}, 2 \pi]$.

Define
\[
Q = \{ r e^{i \theta} \in \overline{D} \mid \tfrac{1}{2} \leq r \leq 1 \},
\]
\[
Q^{+} = \{ r e^{i \theta} \in Q \mid 0 \leq \theta \leq \theta_{0} \}
\qquad \text{and} \qquad
Q^{-} = \{ r e^{i \theta} \in Q \mid \theta_{0} \leq \theta \leq 2 \pi \}.
\]
Since $\xi(D(\frac{1}{2})) = \{c\}$, we have
$\xi^{*} d \alpha_{0}$  vanishes on $D(\frac{1}{2})$ and
\[
\int_{\xi(\overline{D})} d \alpha_{0}
= \int_{\overline{D}} \xi^{*} d \alpha_{0}
= \int_{Q} \xi^{*} d \alpha_{0}.
\]

Now $Q^{\pm}$ is a compact manifold with corners,
and $\xi$ is smooth on $Q^{\pm}$. Applying Stokes' formula
on $Q^{\pm}$, canceling line integrals on $\partial Q^{+} \cap \partial Q^{-}$, we get
\[
\int_{\xi(\overline{D})} d \alpha_{0}
= \int_{\partial Q} \xi^{*} \alpha_{0}
= \int_{\partial \overline{D}} \xi^{*} \alpha_{0}
- \int_{\partial D(\frac{1}{2})} \xi^{*} \alpha_{0}.
\]
Since $\xi^{*} \alpha_{0}$  vanishes on $\partial D(\frac{1}{2})$, we get
(\ref{lemma_isoperimetric_disk_4}) again.

In summary, we get
\[
|\partial \Omega|_{G}^{2} \geq \min \left \{ \frac{C_{1}^{2}}{|\check{D}|_{G}}, C_{4} \right \}
|\Omega|_{G}.
\]
Defining $\mu = \frac{1}{2\pi} \min \left \{ \frac{C_{1}^{2}}{|\check{D}|_{G}}, C_{4} \right \}$,
we finish the proof.
\end{proof}

For the closed unit disk $\overline{D}$ and $\Omega$ as in (\ref{disk_in_disk}), let
\begin{equation}\label{holomorphic_isomorphism}
\varphi: \overline{D} \rightarrow \Omega
\end{equation}
be an arbitrary holomorphic isomorphism.

Let $r_{*}= |\varphi^{-1}(0)|$. Since a M\"{o}bius
transformation maps a disk to a disk, we infer that $\varphi(D(t))$ is a disk in
$\Omega$. Furthermore, if $t \neq r_{*}$, then $0 \notin \partial [\varphi(D(t))]$
and Lemma \ref{lemma_isoperimetric_disk} can be applied to $\varphi(D(t))$.
Note that $G$ is a conformal $C^{2}$ metric on $\Omega - \{0\}$. The function
$\| d \varphi \|$ is undefined at $\varphi^{-1} (0)$ and $C^{2}$ elsewhere.
Therefore, by the comment below (\ref{gromov_formula}), we can apply (\ref{gromov_formula})
to $\varphi$.

Similar to the proof of Lemma \ref{lemma_bound_disk_covering}, we obtain the following
decaying rate (compare \cite[p.\ 91]{mcduff_salamon}).

\begin{lemma}\label{lemma_decay_rate}
For $0< r_{1} \leq r_{2} \leq 1$ and $\varphi$ in (\ref{holomorphic_isomorphism}), we have
\[
|\varphi (D(r_{1}))|_{G} \leq
\left( \frac{r_{1}}{r_{2}} \right)^{\mu}
|\varphi (D(r_{2}))|_{G},
\]
where $\mu >0$ is the constant in Lemma \ref{lemma_isoperimetric_disk}.
\end{lemma}
\begin{proof}
Denote $A(t) = |\varphi (D(t))|_{G}$ and $L(t) = |\varphi (\partial D(t))|_{G}$.
For $t \neq r_{*}$ and $t > 0$, by (\ref{gromov_formula}) and Lemma
\ref{lemma_isoperimetric_disk}, we get
\[
\dot{A} \geq \mu t^{-1} A.
\]
Since $A > 0$ for $t > 0$, we get
\begin{equation}\label{lemma_decay_rate_1}
A^{-1} \dot{A} \geq \mu t^{-1}.
\end{equation}

Since $A$ is continuous on $[0,1]$, the fundamental theorem of calculus still
holds for $\frac{d}{d t} \log A = A^{-1} \dot{A}$ on $[r_{1}, r_{2}]$.
Integrating both sides of (\ref{lemma_decay_rate_1}) on $[r_{1}, r_{2}]$,
we finish the proof.
\end{proof}

Consider $\|d\varphi (te^{i \theta})\|_{G}$ as a function of $(t,\theta)$.
We shall estimate the integral of $\|d\varphi (te^{i \theta})\|_{G}$ on
$(t,\theta) \in [0,r] \times [0, 2\pi]$. Using this estimate, we will derive useful
geometric consequences. Before doing this, we shall prove a lemma
which is its analytic translation.

\begin{lemma}\label{lemma_analytic}
Suppose $\lambda(t)$ is a nonnegative measurable function on $[0,a]$. Suppose
there exist constants $C>0$ and $\nu >0$ such that, for all $r_{1}$ and $r_{2}$
with $0 < r_{1} \leq r_{2} \leq a$,
we have
\[
\int_{0}^{r_{1}} \lambda^{2} t dt \leq C
\left( \frac{r_{1}}{r_{2}} \right)^{\nu}
\int_{0}^{r_{2}} \lambda^{2} t dt.
\]
Then, for all $r \in [0,a]$,
\[
\int_{0}^{r} \lambda dt \leq C^{\frac{1}{2}}
(1 - 2^{-\frac{\nu}{2}})^{-1} (\log 2)^{\frac{1}{2}}
\left( \int_{0}^{r} \lambda^{2} t dt \right)^{\frac{1}{2}}.
\]
\end{lemma}
\begin{proof}
For all integers $k \geq 0$, we have
\begin{eqnarray*}
\left( \int_{2^{-k-1}r}^{2^{-k}r} \lambda dt \right)^{2}
& \leq & \int_{2^{-k-1}r}^{2^{-k}r} \lambda^{2} t dt \cdot
\int_{2^{-k-1}r}^{2^{-k}r} t^{-1} dt \\
& = & (\log 2) \int_{2^{-k-1}r}^{2^{-k}r} \lambda^{2} t dt.
\end{eqnarray*}
Since
\[
\int_{2^{-k-1}r}^{2^{-k}r} \lambda^{2} t dt
 \leq  \int_{0}^{2^{-k}r} \lambda^{2} t dt
 \leq  C (2^{-k})^{\nu} \int_{0}^{r} \lambda^{2} t dt,
\]
we get
\[
\left( \int_{2^{-k-1}r}^{2^{-k}r} \lambda dt \right)^{2}
\leq C (\log 2) (2^{-k})^{\nu} \int_{0}^{r} \lambda^{2} t dt.
\]
Therefore,
\begin{eqnarray*}
\int_{0}^{r} \lambda dt
& = & \sum_{k=0}^{\infty} \int_{2^{-k-1}r}^{2^{-k}r} \lambda dt \\
& \leq & \sum_{k=0}^{\infty} C^{\frac{1}{2}} (\log 2)^{\frac{1}{2}}
(2^{-\frac{\nu}{2}})^{k} \left( \int_{0}^{r} \lambda^{2} t dt \right)^{\frac{1}{2}} \\
& = & C^{\frac{1}{2}} (1 - 2^{-\frac{\nu}{2}})^{-1} (\log 2)^{\frac{1}{2}}
\left( \int_{0}^{r} \lambda^{2} t dt \right)^{\frac{1}{2}}.
\end{eqnarray*}
\end{proof}

\begin{lemma}\label{lemma_derivative_area}
For the $\varphi$ in (\ref{holomorphic_isomorphism}), we have
\[
\int_{0}^{r} \int_{0}^{2 \pi} \|d \varphi (te^{i \theta}) \|_{G} d \theta dt
\leq C(\mu) |\varphi(D(r))|_{G}^{\frac{1}{2}},
\]
where
\[
C(\mu) = (2 \pi)^{\frac{1}{2}}
(1 - 2^{-\frac{\mu}{2}})^{-1} (\log 2)^{\frac{1}{2}},
\]
and $\mu >0$ is the constant in Lemma \ref{lemma_isoperimetric_disk}.
\end{lemma}
\begin{proof}
Define
\[
\lambda(t) = \left( \int_{0}^{2 \pi} \|d \varphi (te^{i \theta}) \|_{G}^{2} d \theta \right)^{\frac{1}{2}},
\]
then
\[
\int_{0}^{r} \lambda^{2} t dt =
\int_{0}^{r} \int_{0}^{2 \pi} \|d \varphi\|_{G}^{2} t d \theta dt
= \int_{D(r)} \|d \varphi\|_{G}^{2}.
\]

By Lemma \ref{lemma_decay_rate}, $\lambda(t)$
satisfies the assumption of Lemma \ref{lemma_analytic} together with $C=1$ and $\nu = \mu$.
Applying Lemma \ref{lemma_analytic}, we get
\begin{eqnarray*}
\int_{0}^{r} \int_{0}^{2 \pi} \|d \varphi\|_{G} d \theta dt
& \leq & \int_{0}^{r} \left( \int_{0}^{2 \pi}
\|d \varphi\|_{G}^{2} d \theta \right)^{\frac{1}{2}} (2 \pi)^{\frac{1}{2}} dt \\
& \leq & (2 \pi)^{\frac{1}{2}}
(1 - 2^{-\frac{\mu}{2}})^{-1} (\log 2)^{\frac{1}{2}}
\left( \int_{D(r)} \|d \varphi\|_{G}^{2} \right)^{\frac{1}{2}},
\end{eqnarray*}
which finishes the proof.
\end{proof}

\begin{corollary}\label{corollary_ray}
For $\varphi$ defined in (\ref{holomorphic_isomorphism}),
we have the following conclusion.

There exists a measurable subset $\Theta \subseteq [0,2 \pi]$ such that
$\Theta$ has a positive measure, and for each $\theta_{0} \in \Theta$,
the length of the curve has the bound
\[
|\varphi(\{ te^{i \theta_{0}} \mid 0 \leq t \leq r \})|_{G}
\leq (2 \pi)^{-1} C(\mu) |\varphi(D(r))|_{G}^{\frac{1}{2}},
\]
where $C(\mu)$ is the constant in Lemma \ref{lemma_derivative_area}.
\end{corollary}
\begin{proof}
Denote $|\varphi(\{ te^{i \theta} \mid 0 \leq t \leq r \})|_{G}$ by
$\mathfrak{l}(\theta)$. Then
\[
\int_{0}^{2 \pi} \mathfrak{l}(\theta) d \theta =
\int_{0}^{2 \pi} \int_{0}^{r} \|d \varphi \|_{G} dt d \theta.
\]
Thus Lemma \ref{lemma_derivative_area} immediately implies this corollary.
\end{proof}

The above lemmas lead to the following important geometric consequence
which is motivated by lines 2-3 in \cite[p.\ 318]{gromov}. In particular,
it implies that the diameter of $F(\check{D}(r))$ is bounded in terms of
the length of $F(\partial D(r))$.

\begin{lemma}\label{lemma_diameter_boundary}
There exists a constant $C>0$ such that the following holds.

For any disk $\Omega$ in (\ref{disk_in_disk}), if $0 \notin \partial \Omega$,
then
\[
d(\Omega)_{G} \leq C |\partial \Omega|_{G},
\]
where $d(\Omega)_{G}$ is the diameter of $\Omega - \{0\}$
with respect to $G$.
\end{lemma}
\begin{proof}
Let $z_{0}$ be an interior point of $\Omega$ such that $z_{0} \neq 0$.
Choose $\varphi$ in (\ref{holomorphic_isomorphism})
such that $\varphi(0) = z_{0}$.

Since $\varphi^{-1}(0) \neq 0$, by Corollary \ref{corollary_ray}, we can find a curve
$\gamma(t) = \varphi(te^{i \theta_{0}})$,
$t \in [0,1]$, such that $0 \notin \gamma([0,1])$ and its length has the bound
\[
|\gamma([0,1])|_{G} \leq
(2 \pi)^{-1} C(\mu) |\varphi(D(1))|_{G}^{\frac{1}{2}}
= (2 \pi)^{-1} C(\mu) |\Omega|_{G}^{\frac{1}{2}}.
\]
Since $0 \notin \partial \Omega$, by Lemma \ref{lemma_isoperimetric_disk}
and the above inequality, we have
\[
|\gamma([0,1])|_{G} \leq (2 \pi)^{-\frac{3}{2}} C(\mu)
\mu^{- \frac{1}{2}} |\partial \Omega|_{G}.
\]

Since $0 \notin \gamma([0,1])$,
we know that $\gamma$ is a curve in $\Omega - \{0\}$ and it connects $z_{0}$ with
$\partial \Omega$.
Thus, for every $z_{0} \in \Omega - \{0\}$, the distance between
$z_{0}$ and $\partial \Omega$ is bounded by $(2 \pi)^{-\frac{3}{2}} C(\mu)
\mu^{- \frac{1}{2}} |\partial \Omega|_{G}$.

Since $\partial \Omega$ is connected, we get
\[
d(\Omega)_{G} \leq 2 (2 \pi)^{-\frac{3}{2}} C(\mu)
\mu^{- \frac{1}{2}} |\partial \Omega|_{G} + |\partial \Omega|_{G}.
\]
This lemma is proved because the constant $\mu$ is independent of $\Omega$, .
\end{proof}

\begin{lemma}\label{lemma_continuity}
The map $F: \check{D} \rightarrow U$ has a continuous extension over $0 \in D$.
\end{lemma}
\begin{proof}
Since $F(\check{D})$ is relatively compact, $\text{Re}H$ and
$\text{Re}\widetilde{H}$ are equivalent on $F(\check{D})$.
By Lemma \ref{lemma_structure_F}, there exists a constant $C_{1} > 0$
such that the diameter of $F(\check{D}(r))$ with respect to $\text{Re}H$
is bounded by $C_{1} d(D(r))_{G}$, where $\check{D}(r) = D(r)-\{0\}$ and
$d(D(r))_{G}$ is the diameter of $\check{D}(r)$ with respect to $G$.

By Lemma \ref{lemma_diameter_boundary}, $d(D(r))_{G}$ is bounded by
$C |\partial D(r)|_{G}$.

By Lemma \ref{lemma_circle_zero}, $|\partial D(r)|_{G} \rightarrow 0$
when $r \rightarrow 0$. Thus the diameter of $F(\check{D}(r))$ shrinks to $0$
when $r \rightarrow 0$. Since $F(\check{D})$ is relatively compact, by Cauchy's
criterion, the limit of $F(z)$ exists when $z \rightarrow 0$, which completes
the proof.
\end{proof}

\section{Lipschitz Continuity}\label{section_lipschitz_continuity}
The goal of this section is to prove the Lipschitz continuity of $f$ near $0 \in D^{+}$.
This step is one of the main differences between a geometric proof and an analytic
proof. More precisely, we shall prove the following lemma.

\begin{lemma}\label{lemma_lipschitz}
Under the assumption of Theorem \ref{theorem_finite_area},
consider the standard Euclidean metric on $\check{D}^{+}$. Then $\|df\|$ is bounded in
$\check{D}^{+}(r)$ for some $r \in (0,1)$. In particular, f
has a Lipschitz continuous extension over $0 \in D^{+}$.
\end{lemma}

As seen from the above lemma, the main task in this section is to estimate derivatives.
We shall apply an argument similar to the proof of Lemma \ref{lemma_bound_disk_covering}.
Typically, we need an isoperimetric inequality stronger than
Lemma \ref{lemma_isoperimetric_disk}.

In order to get such an inequality, we need another
simple and powerful observation due to Gromov \cite[1.3.B]{gromov}: An almost complex
manifold is locally tamed by an exact form which is ``almost" compatible with the
almost complex structure. The idea is as follows. Let $H$ be a Hermitian metric on
$(M,J)$. Denote by $H_{p}$ the Hermitian at $p$, where $p \in B$ and $B$ is contained in a
coordinate chart. Fixing $p$, by using the coordinate in $B$,
we can also consider $-\text{Im}H_{p}$ as a differential form in $B$.
As a constant form, $-\text{Im}H_{p}$ is exact in $B$. Since
$-\text{Im}H_{p}$ is compatible with $J$ at $p$, by the continuity of $J$, the form
$-\text{Im}H_{p}$ tames and is even ``almost" compatible with $J$ near $p$.

\begin{remark}\label{remark_monotone}
The comment below Lemma \ref{lemma_monotone} mentions that $J$ is locally tamed
by exact forms. One can use the above construction to obtain these exact forms.
\end{remark}

In this section, we use the same assumption as in Section \ref{section_continuity}.
Therefore, every result in Section \ref{section_continuity} can be used now.

By Lemma \ref{lemma_continuity}, we know that $F(0)$ is defined and $F(0)$ is in $U$.
We can choose in a chart a neighborhood $B \subseteq U$ of $F(0)$, which is
identified through the coordinates with a closed ball in $\mathbb{R}^{2n}$.

Denote by $J_{p}$, $\widetilde{J}_{p}$, $H_{p}$ and $\widetilde{H}_{p}$ the almost complex structures
or Hermitians at $p \in B$. They are functions on $B$ whose values are matrices (operators and bilinear forms).
Since $B \subseteq \mathbb{R}^{2n}$, the tangent spaces $T_{p}B$ and $T_{q}B$ are
both naturally identified with $\mathbb{R}^{2n}$ for $q \in B$. Therefore,
$J_{p}$, $\widetilde{J}_{p}$, $H_{p}$ and $\widetilde{H}_{p}$ are also defined
on $T_{q}B$.

Denote the Euclidean metric on $\mathbb{R}^{2n}$ by $\langle \cdot, \cdot
\rangle_{\text{Eu}}$ and $\| \cdot \|_{\text{Eu}}$.

We consider $H_{p}(v, J_{q}v)$ and $\widetilde{H}_{p}(v, \widetilde{J}_{q}v)$
as smooth functions of $(p,q,v)$, where $v \in \mathbb{R}^{2n}$.
Since $-\text{Im}H_{p}(v, J_{p} v) > 0$
and $-\text{Im}\widetilde{H}_{p}(v, \widetilde{J}_{p} v) > 0$ for $v \neq 0$,
we can choose $B$ so small that
\[
-\text{Im}H_{p}(v, J_{q} v) \geq C_{0} \langle v, v \rangle_{\text{Eu}}
\]
and
\[
-\text{Im}\widetilde{H}_{p}(v, \widetilde{J}_{q} v) \geq C_{0} \langle v, v \rangle_{\text{Eu}}
\]
for some constant $C_{0}>0$ and all $p$ and $q$ in $B$.

\begin{lemma}\label{lemma_lipschitz_form}
There exists a constant $C>0$ such that, for all $p$, $q \in B$
and $v \in \mathbb{R}^{2n}$, we have
\[
-\text{Im}H_{q} (v, J_{q} v) \leq (1 + C \|p-q\|_{\text{Eu}})
(-\text{Im}H_{p}) (v, J_{q} v),
\]
\[
-\text{Im}\widetilde{H}_{q} (v, \widetilde{J}_{q} v) \leq (1 + C \|p-q\|_{\text{Eu}})
(-\text{Im}\widetilde{H}_{p}) (v, \widetilde{J}_{q} v),
\]
\[
\text{Re}H_{q} (v, v) \leq (1 + C \|p-q\|_{\text{Eu}})
\text{Re}H_{p} (v, v),
\]
and
\[
\text{Re}\widetilde{H}_{q} (v, v) \leq (1 + C \|p-q\|_{\text{Eu}})
\text{Re}\widetilde{H}_{p} (v, v).
\]
\end{lemma}
\begin{proof}
Since $B$ is compact, the derivative of $-\text{Im}H_{p}$
with respect to $p$ is bounded on $B$. Since $B$ is also convex,
applying the fundamental theorem of calculus to $-\text{Im}H_{p}$,
we infer that there exists a constant $C_{1} > 0$ such that
\[
|\text{Im}H_{p} (v, J_{q} v) - \text{Im}H_{q} (v, J_{q} v)|
\leq C_{1} \|p-q\|_{\text{Eu}} \|J_{q}\|_{\text{Eu}}
\|v\|_{\text{Eu}}^{2}.
\]
Clearly, $\|J_{q}\|_{\text{Eu}}$ is also bounded on $B$.

As mentioned above,
\[
-\text{Im}H_{p}(v, J_{q} v) \geq C_{0} \|v\|_{\text{Eu}}^{2}.
\]

Therefore, there exists a constant $C>0$ such that
\[
|\text{Im}H_{p} (v, J_{q} v) - \text{Im}H_{q} (v, J_{q} v)|
\leq C \|p-q\|_{\text{Eu}} (-\text{Im}H_{p}) (v, J_{q} v).
\]
and
\begin{eqnarray*}
-\text{Im}H_{q} (v, J_{q} v)
& \leq & -\text{Im}H_{p} (v, J_{q} v) +
|\text{Im}H_{p} (v, J_{q} v) - \text{Im}H_{q} (v, J_{q} v)| \\
& \leq & (1 + C \|p-q\|_{\text{Eu}}) (-\text{Im}H_{p}) (v, J_{q} v),
\end{eqnarray*}
which proves the first inequality.

By similar arguments, one can prove the other inequalities.
\end{proof}

By Lemma \ref{lemma_continuity}, $F$ is continuous.
Since $B$ is a neighborhood of $F(0)$,
there exists $r_{0}>0$ such that $F(D(r_{0})) \subseteq B$ and
the diameter of $F(D(r_{0}))$ with respect to $\|\cdot\|_{\text{Eu}}$
is less than $1$.

We shall improve Lemma \ref{lemma_isoperimetric_disk} to be
the following lemma on disks inside $D(r_{0})$
(compare \cite[p.\ 317, (12)]{gromov}). The proof is a refinement of
that of Lemma \ref{lemma_isoperimetric_disk}. Besides Gromov's idea described
above, Lemma \ref{lemma_structrue_in_u} needs to be fully employed.

\begin{lemma}\label{lemma_isoperimetric_diameter}
There exists a constant $C>0$ such that the following holds.
Suppose $\Omega$ is a disk in (\ref{disk_in_disk}) such that
$\Omega \subseteq D(r_{0})$ and $0 \notin \partial \Omega$, then
\[
|\partial \Omega|_{G}^{2} \geq 4 \pi (1 - C d(\Omega)_{\text{Eu}}) |\Omega|_{G},
\]
where $d(\Omega)_{\text{Eu}}$ is the diameter of $F(\Omega-\{0\})$ with respect
to $\| \cdot \|_{\text{Eu}}$.
\end{lemma}
\begin{proof}
We only give a detailed proof for the case that $\Omega$ is contained in neither
$D^{+}$ nor $D^{-}$. The proof for the remaining case is similar and even easier.

Define $\Omega^{+} = \Omega \cap \check{D}^{+}$
and $\Omega^{-} = \Omega \cap \check{D}^{-}$. By Lemma \ref{lemma_structure_F}, we get
\[
|\Omega|_{G} = \int_{F(\Omega^{+})} -\text{Im}H
+ \int_{F(\Omega^{-})} -\text{Im}\widetilde{H}.
\]

Choose an arbitrary point $p_{0} \in F(\Omega-\{0\}) \cap W$ and fix it. By the first
two inequalities of Lemma \ref{lemma_lipschitz_form}, there exists a constant $C_{1}>0$
such that
\begin{eqnarray*}
& & \int_{F(\Omega^{+})} -\text{Im}H
+ \int_{F(\Omega^{-})} -\text{Im}\widetilde{H} \\
& \leq & (1 + C_{1} d(\Omega)_{\text{Eu}})
\left( \int_{F(\Omega^{+})} -\text{Im}H_{p_{0}}
+ \int_{F(\Omega^{+})} -\text{Im}\widetilde{H}_{p_{0}} \right).
\end{eqnarray*}
Since $p_{0} \in W$, by Lemma \ref{lemma_structrue_in_u}, we get
$-\text{Im}H_{p_{0}} = -\text{Im}\widetilde{H}_{p_{0}}$. Therefore,
\begin{equation}\label{lemma_isoperimetric_diameter_1}
|\Omega|_{G} \leq (1 + C_{1} d(\Omega)_{\text{Eu}})
\int_{F(\Omega)} -\text{Im}H_{p_{0}}.
\end{equation}

Since $-\text{Im}H_{p_{0}}$ is a constant form in $B$, it has a primitive form
$\alpha_{p_{0}}$. Since $0 \notin \partial \Omega$, similar to the proof of
(\ref{lemma_form_integral_2}), we get
\begin{equation}\label{lemma_isoperimetric_diameter_2}
\int_{F(\Omega)} -\text{Im}H_{p_{0}} = \int_{F(\partial \Omega)} \alpha_{p_{0}}.
\end{equation}

Similar to the proof of Lemma \ref{lemma_isoperimetric_disk}, we construct the cone $K$
described in Section \ref{section_set_up}.
The boundary of $K$ is $F(\partial \Omega)$. As in (\ref{mass_center}), the vertex of $K$
is the center of mass of $F(\partial \Omega)$ with respect to $\text{Re}H_{p_{0}}$.
Construct a map $\xi: \overline{D} \rightarrow K$ as in (\ref{lemma_isoperimetric_disk_2}).
Then by Lemma \ref{lemma_cone_isoperimetric}, we get
\begin{equation}\label{lemma_isoperimetric_diameter_3}
|\xi(\overline{D})|_{\text{Re}H_{p_{0}}} \leq (4 \pi)^{-1}
|F(\partial \Omega)|_{\text{Re}H_{p_{0}}}^{2}.
\end{equation}
Similar to the argument in Lemma \ref{lemma_isoperimetric_disk}, we also get
\begin{equation}\label{lemma_isoperimetric_diameter_4}
\int_{F(\partial \Omega)} \alpha_{p_{0}}
= \int_{\xi(\overline{D})} d \alpha_{p_{0}}
= \int_{\xi(\overline{D})} -\text{Im}H_{p_{0}}.
\end{equation}

Since $-\text{Im}H_{p_{0}}$ is compatible with $J_{p_{0}}$ and $\text{Re}H_{p_{0}}$,
we have Wirtinger's inequality
\[
|-\text{Im}H_{p_{0}} (v_{1}, v_{2})| \leq \|v_{1}\|_{\text{Re}H_{p_{0}}}
\|v_{2}\|_{\text{Re}H_{p_{0}}}.
\]
Applying Wirtinger's inequality, we get
\begin{equation}\label{lemma_isoperimetric_diameter_5}
\int_{\xi(\overline{D})} -\text{Im}H_{p_{0}} \leq |\xi(\overline{D})|_{\text{Re}H_{p_{0}}}.
\end{equation}

By the third and fourth inequalities in Lemma \ref{lemma_lipschitz_form},
for the same constant $C_{1}$ as the above, we have
\begin{eqnarray*}
|F(\partial \Omega \cap \check{D}^{+})|_{\text{Re}H_{p_{0}}} & \leq &
(1 + C_{1} d(\Omega)_{\text{Eu}})^{\frac{1}{2}}
|F(\partial \Omega \cap \check{D}^{+})|_{\text{Re}H} \\
& = & (1 + C_{1} d(\Omega)_{\text{Eu}})^{\frac{1}{2}}
|\partial \Omega \cap \check{D}^{+}|_{G}.
\end{eqnarray*}
Similarly,
\[
|F(\partial \Omega \cap \check{D}^{-})|_{\text{Re}\widetilde{H}_{p_{0}}}
\leq (1 + C_{1} d(\Omega)_{\text{Eu}})^{\frac{1}{2}}
|\partial \Omega \cap \check{D}^{-}|_{G}.
\]
By Lemma \ref{lemma_structrue_in_u} again and the above two inequalities, we get
\begin{eqnarray}\label{lemma_isoperimetric_diameter_6}
|F(\partial \Omega)|_{\text{Re}H_{p_{0}}} & = &
|F(\partial \Omega \cap \check{D}^{+})|_{\text{Re}H_{p_{0}}}
+ |F(\partial \Omega \cap \check{D}^{-})|_{\text{Re}\widetilde{H}_{p_{0}}} \\
& \leq & (1 + C_{1} d(\Omega)_{\text{Eu}})^{\frac{1}{2}}
|\partial \Omega|_{G}. \nonumber
\end{eqnarray}

Combining (\ref{lemma_isoperimetric_diameter_1}), (\ref{lemma_isoperimetric_diameter_2}),
(\ref{lemma_isoperimetric_diameter_4}), (\ref{lemma_isoperimetric_diameter_5}),
(\ref{lemma_isoperimetric_diameter_3}) and (\ref{lemma_isoperimetric_diameter_6}),
we get
\[
|\Omega|_{G} \leq (1 + C_{1} d(\Omega)_{\text{Eu}})^{2} (4 \pi)^{-1}
|\partial \Omega|_{G}^{2}.
\]
Here
\[
(1 + C_{1} d(\Omega)_{\text{Eu}})^{2}
= 1 + 2 C_{1} d(\Omega)_{\text{Eu}} + C_{1}^{2} d(\Omega)_{\text{Eu}}^{2}
\leq 1+ (2 C_{1} + C_{1}^{2}) d(\Omega)_{\text{Eu}},
\]
the last inequality comes from the fact that
the diameter of $F(D(r_{0})-\{0\})$ with respect to
$\| \cdot \|_{\text{Eu}}$ is less than $1$.

Let $C = 2 C_{1} + C_{1}^{2}$, we get
\[
|\partial \Omega|_{G}^{2} \geq 4 \pi (1 + C d(\Omega)_{\text{Eu}})^{-1} |\Omega|_{G}
\geq 4 \pi (1 - C d(\Omega)_{\text{Eu}}) |\Omega|_{G},
\]
which proves the lemma for the case that $F(\Omega)$ is contained in neither
$D^{+}$ nor $D^{-}$.

If $F(\Omega)$ is contained in $D^{+}$ (resp.\ $D^{-}$), then we choose an arbitrary point
$p_{0} \in F(\Omega)$ and compare $H$ (resp.\ $\widetilde{H}$) with $H_{p_{0}}$
(resp.\ $\widetilde{H}_{p_{0}}$). By a similar and easier argument, we finish the proof.
\end{proof}

By Lemma \ref{lemma_isoperimetric_diameter}, we easily get the following isoperimetric
inequality which is the most important for this section. (Compare \cite[p.\ 317, (12')]{gromov}.)

\begin{lemma}\label{lemma_isoperimetric_boundary}
There exists a constant $C>0$ such that the following holds.
For every disk $\Omega$ in Lemma \ref{lemma_isoperimetric_diameter}, we have
\[
|\partial \Omega|_{G}^{2} \geq 4 \pi (1 - C |\partial \Omega|_{G}) |\Omega|_{G}.
\]
\end{lemma}
\begin{proof}
Since $B$ is compact, $\text{Re}H$, $\text{Re}\widetilde{H}$ and $\| \cdot \|_{\text{Eu}}$
are equivalent on $B$. Since $F(D(r_{0})) \subseteq B$, by Lemma \ref{lemma_structure_F},
there exists $C_{1}>0$ such that
\[
d(\Omega)_{\text{Eu}} \leq C_{1} d(\Omega)_{G},
\]
where $d(\Omega)_{\text{Eu}}$ is defined in Lemma \ref{lemma_isoperimetric_diameter}
and $d(\Omega)_{G}$ is defined in Lemma \ref{lemma_diameter_boundary}.

Thus Lemmas \ref{lemma_diameter_boundary} and \ref{lemma_isoperimetric_diameter}
finish the proof.
\end{proof}

The following lemma immediately implies Lemma \ref{lemma_lipschitz} as Lemma
\ref{lemma_metric_hyperbolic_bound} implies Lemma \ref{lemma_tame_to_area}.
The proof follows lines 4-6 in \cite[p.\ 318]{gromov}. It needs an argument
similar to that of Lemma \ref{lemma_bound_disk_covering}.

\begin{lemma}\label{lemma_metric_bound}
Using the standard coordinate $z=x+iy$ on $\mathbb{C}$, the metric $G$ on $\check{D}$
has the form $g dx \otimes dx + g dy \otimes dy$, and $g$ is a bounded function on
$D(\frac{r_{0}}{2}) - \{0\}$.
\end{lemma}
\begin{proof}
Let $z_{0}$ be a point in $D(\frac{r_{0}}{2})$ such that $z_{0} \neq 0$.
Consider a holomorphic isomorphism $\varphi_{z_{0}}: \overline{D} \rightarrow D(r_{0})$
such that $\varphi_{z_{0}}(0) = z_{0}$. Let $r_{*} = |\varphi_{z_{0}}^{-1}(0)|$.

Define $A(t) = |\varphi_{z_{0}}(D(t))|_{G}$ and $L(t) = |\varphi_{z_{0}}(\partial D(t))|_{G}$.
The following computation is similar to the proof of Lemma \ref{lemma_decay_rate}.

By (\ref{gromov_formula}) and Lemma \ref{lemma_isoperimetric_boundary},
there exists $C_{1} > 0$ such that, for $r \in (0,1)$ and $r \neq r_{*}$,
\[
\dot{A} \geq 2 t^{-1} (1- C_{1}L) A.
\]
Since $A > 0$ when $t >0$, we have
\[
A^{-1} \dot{A} \geq 2 t^{-1}- 2 C_{1} t^{-1} L.
\]
Integrating both sides of the above inequalities on $[r,1]$ for $r>0$, we get
\begin{equation}\label{lemma_metric_bound_1}
A(r) \leq r^{2} A(1) \exp \left( 2 C_{1} \int_{r}^{1} t^{-1} L dt \right).
\end{equation}

We have
\begin{eqnarray*}
\int_{r}^{1} t^{-1} L dt & = &
\int_{r}^{1} t^{-1} \int_{0}^{2 \pi} \|d \varphi_{z_{0}} \|_{G}\ t d \theta dt \\
& \leq & \int_{0}^{1}\int_{0}^{2 \pi}
\| d \varphi_{z_{0}} \|_{G} d \theta dt.
\end{eqnarray*}
Applying Lemma \ref{lemma_derivative_area} to the above inequality, we obtain
\begin{equation}\label{lemma_metric_bound_2}
\int_{r}^{1} t^{-1} L dt
\leq C(\mu) |\varphi_{z_{0}}(D)|_{G}^{\frac{1}{2}}.
\end{equation}

By (\ref{lemma_metric_bound_1}) and (\ref{lemma_metric_bound_2}), we get
\begin{equation}\label{lemma_metric_bound_3}
A(r) \leq r^{2} |D(r_{0})|_{G} \exp \left( 2 C_{1} C(\mu)
|D(r_{0})|_{G}^{\frac{1}{2}} \right).
\end{equation}

Since $\displaystyle \| d \varphi_{z_{0}} (0) \|_{G}^{2} = \lim_{r \rightarrow 0}
(\pi r^{2})^{-1} A(r)$, by (\ref{lemma_metric_bound_3}), we infer that
$\| d \varphi_{z_{0}} (0) \|_{G}$ is bounded for all $z_{0} \in D(\frac{r_{0}}{2})$.

When we consider $\varphi_{z_{0}}$ as a holomorphic function on $D$,
denote by $\varphi_{z_{0}}'(0)$ the derivative of $\varphi_{z_{0}}$ at $0$. Then
\[
\|d \varphi_{z_{0}} (0)\|_{G} =  \sqrt{g(z_{0})} |\varphi_{z_{0}}'(0)|.
\]
Since $z_{0} \in D(\frac{r_{0}}{2})$, it's easy to check that $|\varphi_{z_{0}}'(0)|^{-1}$
is bounded for all $z_{0}$. Therefore, $g$ is bounded on $D(\frac{r_{0}}{2}) - \{0\}$.
\end{proof}

\section{An Almost Complex Structure}\label{section_almost_structure}
Suppose $(N,J)$ is an almost
complex manifold. Then its tangent bundle $TN$ is also a manifold.
In this section, we shall describe a natural structure on $TN$ which makes
$TN$ an almost complex manifold. This almost complex structure will be used
in next section. This structure was constructed in
\cite[Proposition 6.7]{yano_kobayashi} and its properties have been extensively studied in \cite{lempert_szoke}.

Denote by $P$ the projection $P: TN \rightarrow N$. For each $q \in N$, $P^{-1}(q)$ is a complex
linear space. Certainly, $P^{-1}(q)$ is a complex manifold. For $v \in P^{-1}(q)$,
we shall use the pair $(q,v)$ to denote this point in $TN$.

Suppose $h: \Omega \rightarrow N$ is a $J$-holomorphic map, where $\Omega$
is an open subset of $\mathbb{C}$. There are liftings of $h$ defined as $\frac{\partial h}{\partial x}: \Omega
\rightarrow TN$ and $\frac{\partial h}{\partial y}: \Omega \rightarrow TN$
which are derivatives of $h$, i.e.
\[
\frac{\partial h}{\partial x} = dh \cdot \frac{\partial}{\partial x}
\qquad \text{and} \qquad
\frac{\partial h}{\partial y} = dh \cdot \frac{\partial}{\partial y},
\]
and $\frac{\partial}{\partial x}$ and $\frac{\partial}{\partial y}$ are vector
fields on $\Omega$.

We shall define an almost complex structure $J^{(1)}$ on $TN$ such that $\frac{\partial h}{\partial x}$
and $\frac{\partial h}{\partial y}$ are $J$-holomorphic maps. For any coordinate chart
$B$ of $N$, by using its coordinate, $TB$ is identified with $B \times \mathbb{R}^{2n}$, and
$T^{2} B = TT B$ is identified with $B \times \mathbb{R}^{2n} \times \mathbb{R}^{2n}
\times \mathbb{R}^{2n}$. Therefore, a point in $T^{2} B$ is $(q, v, w_{1}, w_{2})$,
where $(q,v) \in TB \subseteq TN$, $w_{1} \in T_{q} B = \mathbb{R}^{2n}$, and $w_{2} \in T_{v} \mathbb{R}^{2n}
= \mathbb{R}^{2n}$.

By this identification, the almost complex structure on $B \subseteq N$ becomes a map
$J: B \rightarrow \text{GL}(\mathbb{R}^{2n})$ such that $J(q)^{2} = - \text{Id}$
and the action of $J$ on $T B$ is
\[
J(q,v) = (q, J(q) v).
\]

For convenience, we introduce the notation $d_{v}J$ which is the directional derivative
along the direction $v$, i.e.\ $d_{v} J = dJ \cdot v$. Clearly,
$d_{v} J $ is a linear transformations on $\mathbb{R}^{2n}$.
The definition of $J^{(1)}$ on $T^{2} B$ is (see also \cite[p.\ 76, (3.5)]{lempert_szoke})
\begin{equation}
J^{(1)}(q, v, w_{1}, w_{2}) = (q, v, J(q) w_{1}, (d_{v} J)(q) w_{1} + J(q) w_{2} ).
\end{equation}

By the definition of $J^{(1)}$, we have
\begin{eqnarray*}
\left( J^{(1)} \right)^{2} (q, v, w_{1}, w_{2}) & = &
J^{(1)} (q, v, J w_{1}, (d_{v} J) w_{1} + J w_{2} ) \\
& = & (q, v, J^{2} w_{1}, (d_{v} J) J w_{1} + J (d_{v} J) w_{1} + J^{2} w_{2}).
\end{eqnarray*}
Clearly, $J^{2} = - \text{Id}$, $J^{2} w_{1} = - w_{1}$ and $J^{2} w_{2} = - w_{2}$.
We also have
\[
(d_{v} J) J w_{1} + J (d_{v} J) w_{1} = (d_{v} J^{2}) w_{1} = (- d_{v} \text{Id}) w_{1} = 0.
\]
Therefore, $\left( J^{(1)} \right)^{2} (q, v, w_{1}, w_{2}) = (q, v, -w_{1}, -w_{2})$.

We infer that $J^{(1)}$ defines an almost complex structure in a coordinate chart of $TN$.
Actually, the definition in a particular chart is enough for our proof of Theorems
\ref{theorem_finite_area} and \ref{theorem_tame}.
However, for the interest of readers, we would like to point out that this definition
globally makes sense. (See \cite[Section 3]{lempert_szoke}.)

Moreover, we have following useful proposition (see \cite[Theorem 3.2]{lempert_szoke}).

\begin{proposition}\label{proposition_almost_complex}
The above $J^{(1)}$ is a well defined an almost complex structure on the manifold $TN$
which satisfies the following properties.

(1). The projection $P: TN \rightarrow N$ is $J$-holomorphic.

(2). The inclusion of the fibre $P^{-1}(q) \hookrightarrow TN$ is $J$-holomorphic.

(3). If $\Omega$ is an open subset of $\mathbb{C}$ and $h: \Omega \rightarrow N$ is $J$-holomorphic,
then $\frac{\partial h}{\partial x}:
\Omega \rightarrow TN$ and $\frac{\partial h}{\partial y}: \Omega \rightarrow TN$ are $J$-holomorphic.
\end{proposition}
\begin{proof}
(1) and (2) are obviously true in a coordinate chart. Therefore, they are true globally.

(3) follows immediately from (c) of \cite[Theorem 3.2]{lempert_szoke}. One can also
easily check it by differentiating the equation $\frac{\partial h}{\partial y}
= J \frac{\partial h}{\partial x}$.
\end{proof}

\begin{remark}
In \cite[1.4]{gromov}, Gromov describes an almost complex structure which
is slightly different from the above $J^{(1)}$. But these two structures are
actually related. They also play the same role in the proof of Theorems \ref{theorem_finite_area}
and \ref{theorem_tame}. Let $PTN$ be the projectivization of $TN$. The tangent bundle of
$PTN$ contains a subbundle $\Theta$. (Here $\Theta$ is the notation in \cite[p.\ 318, 1.4]{gromov}.)
The paper \cite{gromov} defines an complex structure on the vector bundle $\Theta \rightarrow PTN$.
In fact, this complex structure comes form the above $J^{(1)}$ on $T^{2}N \rightarrow TN$.
Define $TN-N = \{ (q,v) \mid q \in N, v \in T_{q}N, v \neq 0 \}$. Denote the natural projection
by $P^{(1)}: (TN-N) \rightarrow PTN$. We try to descend $J^{(1)}$ on $TN-N$ to $PTN$. More precisely,
suppose $w \in T_{(q,v)} (TN-N)$, $u \in T_{P^{(1)}(q,v)} PTN$, and $d P^{(1)} w = u$,
we try to define $J^{(1)} u = d P^{(1)} (J^{(1)} w)$. This definition does not
work in general because it depends on the choice of $w$ for a fixed $u$.
However, it does work when $u \in \Theta$. This induces
the complex structure in \cite{gromov} on $\Theta$.
\end{remark}

\section{Higher Order Derivatives}\label{section_higher_order_derivatives}
The goal of this section is to finish the proof of Theorems \ref{theorem_finite_area}
and \ref{theorem_tame}. By Lemma \ref{lemma_lipschitz}, it remains to study the higher
order derivatives of $f$ near $0 \in D^{+}$.

The proof follows Gromov's approach in \cite[1.4]{gromov}. The idea is as follows.
Roughly speaking, based on Lemma \ref{lemma_lipschitz}, the derivatives of $f$ are
also $J$-holomorphic maps which satisfy the same assumption as $f$ does. Therefore, a
boot-strapping argument finishes the proof.

This argument is remarkably different from that of an analytic proof. In an analytic
proof, the argument of higher order regularity is reduced to the problem of the
elliptic regularity of PDEs (see e.g. \cite[p.\ 92 \& Appendix B.4]{mcduff_salamon}).
However, this argument relies on a geometric construction.

First, we recall the definition of a differentiable map on $D^{+}(r)= \{z \mid |z|<r, \text{Im}z \geq 0 \}$,
where $0 < r \leq 1$.
A map $h: D^{+}(r) \rightarrow \mathbb{R}^{m}$ is said to be $C^{k}$
($1 \leq k \leq +\infty$) if, for every $z \in D^{+}(r)$, there exists an open neighborhood
$U_{z}$ of $z$ in $\mathbb{R}^{2}$ and a $C^{k}$ function $h_{z}: U_{z} \rightarrow \mathbb{R}^{m}$
such that $h|_{U_{z} \cap D^{+}(r)} = h_{z}|_{U_{z} \cap D^{+}(r)}$. This is the standard
definition of a $C^{k}$ map on a manifold with boundary.

We shall use the following Lemmas \ref{lemma_derivative_extension} and \ref{lemma_smooth} about $C^{k}$ maps on $D^{+}(r)$.
Their proofs are given in the Appendix.

\begin{lemma}\label{lemma_derivative_extension}
Suppose $h: D^{+}(r) \rightarrow \mathbb{R}^{m}$ is $C^{k}$ and
$h: \check{D}^{+}(r) \rightarrow \mathbb{R}^{m}$ is $C^{k+1}$ for some $k$ such that $0 \leq k < + \infty$.
Suppose $d^{k+1}h$ on $\check{D}^{+}(r)$ has a continuous extension over $0 \in D^{+}(r)$. Then $h$ is
$C^{k+1}$ in $D^{+}(r)$.
\end{lemma}

When we consider a function $h$ defined on an open domain of $\mathbb{R}^{n}$. We say
$h$ is $C^{\infty}$ or smooth if $h$ is $C^{k}$ for all $k$ such that $0 \leq k < + \infty$.
However, when we consider the similar situation on $D^{+}(r)$, the situation becomes slightly
subtle. If $h$ is $C^{k}$ on $D^{+}(r)$ for all $k$ such that $0 \leq k < + \infty$, then
is $h$ a $C^{\infty}$ function? Or can we find a $C^{\infty}$ function $h_{z}$ on $U_{z}$ such that
$h|_{U_{z} \cap D^{+}(r)} = h_{z}|_{U_{z} \cap D^{+}(r)}$ for all $z$? The following
lemma gives the affirmative answer.

\begin{lemma}\label{lemma_smooth}
Suppose $h: D^{+}(r) \rightarrow \mathbb{R}^{m}$ is $C^{k}$ for all $k$ such that
$0 \leq k < + \infty$. Then $h$ is $C^{\infty}$.
\end{lemma}

By Lemma \ref{lemma_lipschitz}, $f$ is continuous at $z=0$.
Since $f(\partial \check{D}^{+}) \subseteq W$ and $W$ is closed,
we have $f(0) \in W$. We can find an open neighborhood $U_{0}$ of $f(0)$
in $M$ satisfying the following properties: (1).\ By (3) in Lemma \ref{lemma_hermitian},
we can require that a $1$-form $\alpha_{0}$ is defined in $U_{0}$,
it vanishes on $W \cap U_{0}$ and
\begin{equation}\label{tame_U_0}
d \alpha_{0}(v,Jv) \geq C_{0} \|v\|^{2}
\end{equation}
for some constant $C_{0} > 0$ and all tangent vectors on $U_{0}$.
(2).\ $U_{0}$ is a coordinate chart, by this coordinate, $U_{0}$ is identified with
$B_{1} \times B_{2} \subseteq \mathbb{R}^{n} \times \mathbb{R}^{n}$, where $B_{1}$
and $B_{2}$ are open subsets of $\mathbb{R}^{n}$, and $W \cap U_{0} = B_{1} \times \{0\}$.
(3).\ There exists a complex frame $\{ e_{1}, \cdots, e_{n} \}$ of $TU_{0}$ such that $e_{j}(q) \in T_{q}W$
($1 \leq j \leq n$) for all $q \in W \cap U_{0}$.

Using the above frame, $TU_{0}$ is identified with $U_{0} \times \mathbb{C}^{n}$,
\begin{equation}\label{T_2_U_0}
T(U_{0} \times \mathbb{C}^{n})= U_{0} \times \mathbb{C}^{n} \times \mathbb{R}^{2n}
\times \mathbb{C}^{n}
\end{equation}
and a point in $T(U_{0} \times \mathbb{C}^{n})$ is represented
by a tuple $(q, v, w_{1}, w_{2})$, where $q \in U_{0}$, $v \in \mathbb{C}^{n}$,
$w_{1} \in T_{q} U_{0} = \mathbb{R}^{2n}$ and $w_{2} \in T_{v} \mathbb{C}^{n}
= \mathbb{C}^{n}$. We use the notation
\begin{equation}
W_{0} = W \cap U_{0}.
\end{equation}
Then $W_{0} = B_{1} \times \{0\}$,
\begin{equation}\label{T_W_0}
TW_{0} = (B_{1} \times \{0\}) \times \mathbb{R}^{n}
\subseteq U_{0} \times \mathbb{C}^{n},
\end{equation}
and $J(TW_{0}) = (B_{1} \times \{0\})
\times J_{0} \mathbb{R}^{n} \subseteq U_{0} \times \mathbb{C}^{n}$,
where $\mathbb{R}^{n} \hookrightarrow \mathbb{C}^{n}$
is the standard inclusion and $J_{0}$ is the standard complex structure on $\mathbb{C}^{n}$.

Define the projections
\begin{equation}
P_{1}: U_{0} \times \mathbb{C}^{n} \rightarrow U_{0}
\qquad \text{and} \qquad
P_{2}: U_{0} \times \mathbb{C}^{n} \rightarrow \mathbb{C}^{n}.
\end{equation}

Let's consider the almost complex structure $J^{(1)}$ on $U_{0} \times \mathbb{C}^{n} = TU_{0}$
which is defined in Proposition \ref{proposition_almost_complex}. By (1) and (2) of
Proposition \ref{proposition_almost_complex}, we know that
$P_{1}$ and the fibre inclusion
$\{q\} \times \mathbb{C}^{n} \hookrightarrow U_{0} \times \mathbb{C}^{n}$ are $J$-holomorphic.
Thus, by (\ref{T_2_U_0}), $J^{(1)}$ on $T^{2} U_{0}$ has the form
\begin{equation}\label{J_1}
J^{(1)} (q, v, w_{1}, w_{2}) =
(q, v, J(q) w_{1}, \phi(q,v) w_{1} + J_{0} w_{2}),
\end{equation}
where $\phi$ is a function on $U_{0} \times \mathbb{C}^{n}$
whose values are real linear maps from $\mathbb{R}^{2n}$ to $\mathbb{C}^{n}$.

By Lemma \ref{lemma_lipschitz}, we know that there exists $\epsilon >0$
such that the following holds: (1).\ $f(D^{+}(\epsilon)) \subseteq U_{0}$,
where $D^{+}(\epsilon) = \{ z \in D^{+} \mid |z| < \epsilon \}$.
(2).\ There exists an open ball $B_{3}$ with finite radius in $\mathbb{C}^{n}$
such that the image of $\frac{\partial f}{\partial x}$
(resp. $\frac{\partial f}{\partial y}$) $: \check{D}^{+}(\epsilon)
\rightarrow TU_{0} = U_{0} \times \mathbb{C}^{n}$ is actually contained in $U_{0} \times B_{3}$,
i.e.\ we get the map
\begin{equation}
\frac{\partial f}{\partial x} \ \left( \text{resp.}\ \frac{\partial f}{\partial y} \right):
\check{D}^{+}(\epsilon)
\rightarrow U_{0} \times B_{3} \subseteq TU_{0}.
\end{equation}
(3).\ The images of $\frac{\partial f}{\partial x}$ and $\frac{\partial f}{\partial y}$
are relatively compact in $U_{0} \times B_{3}$.

Since $B_{3}$ is relatively compact in $\mathbb{C}^{n}$, shrinking $U_{0}$
if necessary, we get
\begin{equation}\label{vertical_bound}
\|\phi(q,v)\| \leq C_{1}
\end{equation}
for some constant $C_{1}>0$, $\phi$ in (\ref{J_1}) and all $(q,v) \in U_{0} \times B_{3}$.
Here we consider $\phi(q,v)$ as a real linear map from
$T_{q}U_{0}$ to $\mathbb{C}^{n}$, $T_{q}U_{0}$ is equipped with the Hermitian $H_{q}$,
$\mathbb{C}^{n}$ is equipped with the standard metric, and the norm of $\phi$ comes
from these two metrics.

Define
\begin{equation}
W_{x} = TW_{0} \cap (U_{0} \times B_{3})
\qquad \text{and} \qquad
W_{y} = J(TW_{0}) \cap (U_{0} \times B_{3}).
\end{equation}

\begin{lemma}
$W_{x}$ and $W_{y}$ are closed totally real submanifolds of $U_{0} \times B_{3}$.
\end{lemma}
\begin{proof}
It suffices to show that $TW_{0}$ and $J(TW_{0})$ are closed totally real submanifolds
of $U_{0} \times \mathbb{C}^{n}$.
Clearly, they are closed submanifolds. It remains to check that they
are totally real (see Definition \ref{definition_totally_real}).

Obviously, we have
$\dim (TW_{0}) = \dim (J(TW_{0})) = \frac{1}{2} \dim (U_{0} \times \mathbb{C}^{n})$.

By (\ref{T_2_U_0}), (\ref{T_W_0}) and (\ref{J_1}),
at the base point $(q,v) \in TW_{0}$, the elements in $T^{2} W_{0}$ and
$J^{(1)} (T^{2} W_{0})$ have the form
\[
(q, v, w_{1}, w_{2}) \in T^{2} W_{0}
\]
and
\[
(q, v, J(q) u_{1}, \phi(q,v)u_{1} + J_{0} u_{2}) \in J^{(1)}(T^{2} W_{0}),
\]
where $w_{1}$, $u_{1} \in \mathbb{R}^{n} \times \{0\} \subseteq \mathbb{R}^{2n}$, and
$w_{2}$, $u_{2} \in \mathbb{R}^{n} \subseteq \mathbb{C}^{n}$. Then
$J(q) u_{1} \in J(q) \mathbb{R}^{n}$ and
$J_{0} u_{2} \in J_{0} \mathbb{R}^{n}$. If
\[
(q, v, w_{1}, w_{2}) =
(q, v, J(q) u_{1}, \phi(q,v)u_{1} + J_{0} u_{2}),
\]
by the fact $J(q) \mathbb{R}^{n} \cap \mathbb{R}^{n} = \{0\}$
and $J_{0} \mathbb{R}^{n} \cap \mathbb{R}^{n} = \{0\}$,
we infer that the vectors $w_{1}$, $w_{2}$, $u_{1}$ and $u_{2}$ are zeros.
This implies that $TW_{0}$ is totally real.

By a similar argument, one sees that $J(TW_{0})$ is also totally real.
\end{proof}

By the fact that $f(\partial \check{D}^{+}) \subseteq W$ and $f$ is
$J$-holomorphic, we know that $\frac{\partial f}{\partial x}
(\partial \check{D}^{+}(\epsilon))$ $\subseteq W_{x}$ and
$\frac{\partial f}{\partial y} (\partial \check{D}^{+}(\epsilon)) \subseteq W_{y}$.
By (3) of Proposition \ref{proposition_almost_complex}, we infer $\frac{\partial f}{\partial x}$
and $\frac{\partial f}{\partial y}$ are $J$-holomorphic in the interior of
$\check{D}^{+}(\epsilon)$. By the smoothness of $\frac{\partial f}{\partial x}$
and $\frac{\partial f}{\partial y}$, we know that they are $J$-holomorphic
on $\check{D}^{+}(\epsilon)$.

Now we define a $1$-form on $U_{0} \times B_{3}$. In $U_{0}$, we already have the
$1$-form $\alpha_{0}$ satisfying (\ref{tame_U_0}).
Using the standard coordinate $(x_{1}+iy_{1}, \cdots, x_{n}+iy_{n}) \in \mathbb{C}^{n}$,
define a $1$-form in $\mathbb{C}^{n}$ as
\begin{equation}
\beta = \sum_{j=1}^{n} x_{j} d y_{j}.
\end{equation}
For $\lambda > 0$, define a $1$-form on $U_{0} \times B_{3}$ as
\begin{equation}\label{form}
\alpha = P_{1}^{*} \alpha_{0} + \lambda P_{2}^{*} \beta.
\end{equation}

\begin{lemma}
The $1$-form $\alpha$ in (\ref{form}) vanishes on $W_{x}$ and $W_{y}$. Furthermore,
there exists a $\lambda > 0$ such that $d \alpha$ tames $J^{(1)}$
on $U_{0} \times B_{3}$.
\end{lemma}
\begin{proof}
Since $\alpha_{0}$ vanishes on $W_{0}$ and $\beta$ vanishes on $\mathbb{R}^{n}$
and $J_{0} \mathbb{R}^{n}$ in $\mathbb{C}^{n}$, we know that $\alpha$ vanishes on
$W_{x}$ and $W_{y}$.

By (\ref{tame_U_0}), (\ref{vertical_bound}) and the fact that
$d \beta (w_{2}, J_{0} w_{2}) = \| w_{2} \|^{2}$ for $w_{2} \in \mathbb{C}^{n}$,
we have
\begin{eqnarray*}
& & d \alpha ((w_{1}, w_{2}), J^{(1)}(w_{1}, w_{2})) \\
& = & d \alpha_{0} (dP_{1}(w_{1}, w_{2}), dP_{1} \cdot J^{(1)}(w_{1}, w_{2}))
+ \lambda d \beta (dP_{2}(w_{1}, w_{2}), dP_{2} \cdot J^{(1)}(w_{1}, w_{2})) \\
& = & d \alpha_{0} (w_{1}, Jw_{1})
+ \lambda d \beta (w_{2}, \phi w_{1} + J_{0} w_{2}) \\
& = & d \alpha_{0} (w_{1}, Jw_{1}) + \lambda d \beta (w_{2}, \phi w_{1})
+ \lambda d \beta (w_{2}, J_{0} w_{2}) \\
& \geq & C_{0} \|w_{1}\|^{2} - \lambda \|w_{2}\| \|\phi\| \|w_{1}\| + \lambda \|w_{2}\|^{2} \\
& \geq & C_{0} \|w_{1}\|^{2} - \lambda C_{1} \|w_{1}\| \|w_{2}\| + \lambda \|w_{2}\|^{2} \\
& \geq & C_{0} \|w_{1}\|^{2} - \lambda C_{1} \left( \tfrac{1}{2} C_{1} \|w_{1}\|^{2}
+ \tfrac{1}{2} C_{1}^{-1} \|w_{2}\|^{2} \right) + \lambda \|w_{2}\|^{2} \\
& = & \left( C_{0} - \tfrac{1}{2} \lambda C_{1}^{2} \right) \|w_{1}\|^{2}
+ \tfrac{1}{2} \lambda \|w_{2}\|^{2}.
\end{eqnarray*}

We finish the proof by choosing $\lambda = C_{0} C_{1}^{-2}$.
\end{proof}

In summary, the maps $\frac{\partial f}{\partial x}: (\check{D}^{+}(\epsilon),
\partial \check{D}^{+}(\epsilon)) \rightarrow (U_{0} \times B_{3}, W_{x})$
and $\frac{\partial f}{\partial y}: (\check{D}^{+}(\epsilon),$
$\partial \check{D}^{+}(\epsilon)) \rightarrow (U_{0} \times B_{3}, W_{y})$
satisfy the assumption of Theorem \ref{theorem_tame}. Applying Lemma
\ref{lemma_tame_to_area} to these two maps, we obtain the following.

\begin{lemma}\label{lemma_boot_strap}
There exists $\epsilon_{0} > 0$ such that
$\frac{\partial f}{\partial x}: (\check{D}^{+}(\epsilon_{0}),
\partial \check{D}^{+}(\epsilon_{0})) \rightarrow (U_{0} \times B_{3}, W_{x})$
and $\frac{\partial f}{\partial y}: (\check{D}^{+}(\epsilon_{0}),
\partial \check{D}^{+}(\epsilon_{0})) \rightarrow (U_{0} \times B_{3}, W_{y})$
are well-defined $J$-holomorphic maps. Here $W_{x}$ and $W_{y}$ are closed totally real submanifolds
of $U_{0} \times B_{3}$. The images of $\frac{\partial f}{\partial x}$ and
$\frac{\partial f}{\partial y}$ are
relatively compact and have finite areas with respect to
any metric on $U_{0} \times B_{3}$.
\end{lemma}

Now we are in a position to finish the proof of Theorems \ref{theorem_finite_area}
and \ref{theorem_tame}. The comment after Lemma \ref{lemma_tame_to_area} tells
us that, by Lemma \ref{lemma_tame_to_area}, Theorem \ref{theorem_finite_area} immediately
implies Theorem \ref{theorem_tame}. Therefore, it suffices to prove Theorem \ref{theorem_finite_area}.

\begin{proof}[Proof of Theorem \ref{theorem_finite_area}]
Following \cite[1.4.B]{gromov}, this proof is a boot-strapping argument.

By Lemma \ref{lemma_boot_strap}, we know that
$\frac{\partial f}{\partial x}$ and $\frac{\partial f}{\partial y}$ are well-defined
on $\check{D}^{+}(\epsilon_{0})$. The proof mainly contains two steps.

First, we prove that, if $f$ has a $C^{k}$ extension for some $k$ such that
$0 \leq k < + \infty$, then $\frac{\partial f}{\partial x}$
and $\frac{\partial f}{\partial y}$ also have $C^{k}$ extensions.

By comparing the conclusion of Lemma \ref{lemma_boot_strap} and the assumption of Theorem \ref{theorem_finite_area},
up to a holomorphic reparametrization of the domain, we infer
$\frac{\partial f}{\partial x}: \check{D}^{+}(\epsilon_{0})
\rightarrow U_{0} \times B_{3}$
and $\frac{\partial f}{\partial y}: \check{D}^{+}(\epsilon_{0})
\rightarrow U_{0} \times B_{3}$ can be viewed as special cases of $f$
because $f$ is the map in the general Theorem \ref{theorem_finite_area}.
Therefore, if $f$ has a $C^{k}$ extension, then, as special cases,
so do $\frac{\partial f}{\partial x}$ and $\frac{\partial f}{\partial y}$.

Second, we prove $f$ is $C^{k}$ on $D^{+}$ for all $k$ such that $0 \leq k < + \infty$.

By Lemma \ref{lemma_continuity}, we know $f$ has a continuous extension over $0 \in D^{+}$.
By taking a coordinate chart near $f(0)$, we may assume that $f$ maps $D^{+}(r)$ into
$\mathbb{R}^{m}$ for some $r \in (0, \epsilon_{0})$.

By the continuity of $f$ on $D^{+}$ and
the result of the first step, we infer that $\frac{\partial f}{\partial x}$ and $\frac{\partial f}{\partial y}$
also have $C^{0}$ extensions. Clearly, $f$ is smooth on $\check{D}^{+}$. By Lemma
\ref{lemma_derivative_extension}, we know that $f$ has a $C^{1}$ extension.

In general, if we know that $f$ has a $C^{k}$ extension over $0$, then, by the result of the first step,
$\frac{\partial f}{\partial x}$ and $\frac{\partial f}{\partial y}$ also
have $C^{k}$ extensions. Since $f$ is smooth on $\check{D}^{+}$, by Lemma \ref{lemma_derivative_extension}, we infer
$f$ has a $C^{k+1}$ extension.

By an induction on $k$, we finished the second step.

Finally, Lemma \ref{lemma_smooth} and the result of the second step finish the proof.
\end{proof}

\appendix
\section{}
In this appendix, we use Whitney's classical result \cite{whitney} to prove
Lemmas \ref{lemma_derivative_extension} and \ref{lemma_smooth}.

\begin{proof}[Proof of Lemma \ref{lemma_derivative_extension}]
It suffices to prove that each coordinate function of $h$ is $C^{k+1}$. Thus
we may assume that $m=1$.

We first prove the case of $k=0$.

Denote by $\Phi$ the continuous extension of $dh$ on $D^{+}(r)$. Then $\Phi$ is a continuous
function from $D^{+}(r)$ to $L(\mathbb{R}^{2}, \mathbb{R}^{1})$, where $L(\mathbb{R}^{2}, \mathbb{R}^{1})$
is the linear space of linear maps from $\mathbb{R}^{2}$ to $\mathbb{R}^{1}$.
Since $D^{+}(r)$ is convex, for all $z_{1}$ and $z_{2}$ in $D^{+}(r)$, the integral
\[
\int_{0}^{1} dh[z_{1} +t (z_{2}-z_{1})] dt =
\int_{0}^{1} \Phi [z_{1} +t (z_{2}-z_{1})] dt
\]
makes sense. Furthermore, since $h$ is continuous on $D^{+}(r)$ and $C^{1}$ on $\check{D}^{+}(r)$,
the fundamental theorem of calculus still holds for $h$ on $D^{+}(r)$. Thus
\begin{eqnarray}\label{lemma_derivative_extension_1}
&   & h(z_{2}) - h(z_{1}) - \Phi(z_{1}) (z_{2}-z_{1}) \\
& = & \int_{0}^{1} \{\Phi [z_{1} +t (z_{2}-z_{1})] - \Phi(z_{1})\} dt \cdot (z_{2}-z_{1}). \nonumber
\end{eqnarray}

Since $\Phi$ is continuous on $D^{+}(r)$, we obtain the following:
for any $z_{0} \in D^{+}(r)$ and for any $\epsilon >0$, there exists $\delta >0$
such that, if $|z_{1} - z_{0}| \leq \delta$ and $|z_{2} - z_{0}| \leq \delta$,
then
\begin{equation}\label{lemma_derivative_extension_2}
\| \Phi (z_{2}) - \Phi(z_{1}) \| \leq \epsilon,
\end{equation}
$\| \Phi [z_{1} +t (z_{2}-z_{1})] - \Phi(z_{1}) \| \leq \epsilon$
for all $t \in [0,1]$ and
\begin{equation}\label{lemma_derivative_extension_3}
|h(z_{2}) - h(z_{1}) - \Phi(z_{1}) (z_{2}-z_{1})| \leq \epsilon |z_{2}-z_{1}|.
\end{equation}

The inequalities (\ref{lemma_derivative_extension_2}) and
(\ref{lemma_derivative_extension_3}) imply that $h$ is a $C^{1}$ function
in the sense of Whitney and the Whitney derivative of $h$ is
$\Phi$ (see \cite[p.64]{whitney}). By Whitney's theorem \cite[Theorem 1]{whitney},
we infer that $h$ is $C^{1}$.

The proof of the case of $k > 0$ is similar to the previous case.

Let $\Phi$ be the continuous extension of $d^{k+1}h$ on $D^{+}(r)$. Similar to
(\ref{lemma_derivative_extension_1}), we can prove the following Taylor expansion. For $0 \leq i \leq k$,
\begin{eqnarray*}
&   & d^{i}h(z_{2}) - \sum_{j=i}^{k} \frac{1}{(j-i)!} d^{j}h(z_{1}) (z_{2}-z_{1})^{j-i}
- \frac{1}{(k+1-i)!} \Phi(z_{1}) (z_{2}-z_{1})^{k+1-i} \\
& = & \frac{1}{(k-i)!} \int_{0}^{1} (1-t)^{k-i} \{ \Phi [z_{1} +t (z_{2}-z_{1})] - \Phi(z_{1}) \} dt
\cdot (z_{2}-z_{1})^{k+1-i},
\end{eqnarray*}
where the $j$-th differential $d^{j}h(z_{1})$ is a multiple linear map with $j$ arguments,
$d^{j}h(z_{1}) (z_{2}-z_{1})^{l}$ ($0 \leq l \leq j$) is the abbreviation of
$d^{j}h(z_{1})(z_{2}-z_{1}, \cdots, z_{2}-z_{1})$, i.e.\ plug $l$ many $z_{2}-z_{1}$ into
the first $l$ arguments of $d^{j}h(z_{1})$.

Since $\Phi$ is continuous, by the above Taylor expansion, we can prove that $h$ is $C^{k+1}$
in the sense of Whitney and the Whitney derivatives of $h$ are $dh, \cdots, d^{k}h$ and
$\Phi$. Then Whitney's theorem again implies that $h$ is $C^{k+1}$.
\end{proof}

\begin{proof}[Proof of Lemma \ref{lemma_smooth}]
It suffices to consider the coordinate function of $h$. Thus we may assume $m=1$.
It's easy to check that $h$ is $C^{\infty}$ in the sense of Whitney \cite[p.65]{whitney}.
By Whitney's theorem \cite[Theorem 1]{whitney}, we finish the proof.
\end{proof}

\section*{Acknowledgements}
The first author would like to thank his PhD advisor Prof.\ Peter Albers
for his support and guidance.
The research of the first author was supported by the NSF grant
DMS~1001701 and the SFB 878-Groups, Geometry and Actions.
The second author wishes to thank his PhD advisor Prof.\ John Klein for suggesting symplectic geometry
as a research area, and for encouragement and support over the years.


\end{document}